\documentclass[11pt]{amsart}
\usepackage{tikz}
\usepackage{xcolor}

\usepackage{hyperref}
\usepackage{standard}
\usepackage{amsmath}
\usepackage{amsrefs}
\usepackage{slashed}
\newcommand{\dirac}{\slashed{\partial}}

\DeclareMathOperator{\specb}{spec_{\bo}}
\newcommand{\grad}{\nabla}
\newcommand{\Diff}{\operatorname{Diff}}
\newcommand{\Sigmadot}{\dot{\Sigma}}
\newcommand{\sigmau}{\underline{\sigma}}
\newcommand{\tauu}{\underline{\tau}}
\newcommand{\etau}{\underline{\eta}}

\newcommand{\dop}{\eth} 
\newcommand{\pot}{\mathbf{A}}
\newcommand{\bA}{\mathbf{A}}
\newcommand{\tL}{\widetilde{L}}
\newcommand{\bF}{\mathbf{F}}
\newcommand{\Diffb}{\Diff_{\bo}}

\newcommand{\tG}{\widetilde{G}}
\newcommand{\bJ}{\mathbf{J}}
\newcommand{\bL}{\mathbf{L}}

\newcommand{\bSigma}{\boldsymbol{\Sigma}}
\newcommand{\br}{\mathbf{r}}
\newcommand{\bp}{\mathbf{p}}

\newcommand{\balpha}{\boldsymbol{\alpha}}
\newcommand{\bgamma}{\boldsymbol{\gamma}}
\newcommand{\bB}{\mathbf{B}}
\newcommand{\bR}{\mathbf{R}}

\newcommand{\bsigma}{\boldsymbol{\sigma}}
\newcommand{\jw}[1]{}
\newcommand{\db}[1]{}

\DeclareMathOperator{\sgn}{sgn}
\newcommand{\module}{\mathcal{M}}

\newcommand\bo{{\mathrm{b}}}
\newcommand\eo{{\mathrm{e}}}

\newcommand{\Tbstar}{{}^{\bo}T^*}
\newcommand{\Testar}{{}^{\eo}T^*}
\newcommand{\Sestar}{{}^{\eo}S^*}
\newcommand{\Sbstar}{{}^{\bo}S^*}

\newcommand\Tb{{}^{\bo} T}
\newcommand\Te{{}^{\eo} T}

\newcommand{\sigmab}{\sigma_\bo}
\newcommand{\Psib}{\Psi_\bo}
\newcommand{\Psie}{\Psi_\eo}

\newcommand{\He}{H_\eo}
\newcommand{\WFb}{\operatorname{WF}_\bo}
\newcommand{\WFe}{\operatorname{WF}_\eo}

\DeclareMathOperator{\liptic}{ell}
\newcommand{\algebra}{\mathcal{A}}
\newcommand{\flowout}{\mathcal{F}}

\DeclareMathOperator{\IC}{\mathsf{IC}}
\DeclareMathOperator{\OG}{\mathsf{OG}}

\newcommand{\charge}{\mathsf{Z}}

\DeclareMathOperator{\Id}{Id}
\newcommand{\bH}{H_{\bo}}
\newcommand{\dom}{\mathcal{D}}

\newcommand{\ellipticdirac}{\mathcal{B}}

\newcommand{\hamvf}{\mathsf{H}}
\newcommand{\sigmae}{\sigma_{\eo}}
\newcommand{\Diffe}{\operatorname{Diff}_{\eo}}

\newcommand{\diff}{\mathsf{D}}

\newcommand{\geom}{\mathsf{G}}

\newcommand{\DiffPsi}{\operatorname{Diff}^2\!\Psi_b}

 \newtheorem{propest}{Propagation Estimate}
\newtheorem*{assumption*}{Klein--Gordon Hypotheses}

\newcommand{\tQ}{\widetilde{Q}}
\newcommand{\tR}{\widetilde{R}}

\newcommand{\reals}{\mathbb{R}}
\newcommand{\pd}[1][]{\partial_{#1}}

\newcommand{\loc}{\text{loc}}
\DeclareMathOperator{\SO}{SO}
\newcommand{\coiso}{\mathcal{C}}
\newcommand{\hilbert}{\mathcal{H}}
\newcommand{\blowdown}{\mathfrak{b}}
\newcommand{\lag}{\mathcal{L}}

\title{Diffraction for the Dirac-Coulomb propagator}
\author{Dean Baskin and Jared Wunsch}

\date{\today}

\thanks{The authors are grateful to Christian G\'erard and Micha\l\
  Wrochna for suggesting the problem and providing helpful insight into
  its importance, as well as for helpful comments on an early version
  of the manuscript.  They are also grateful to Richard Melrose, Andr\'as
  Vasy, and especially Oran Gannot for many helpful conversations.
  The research for this paper began during a Research in Paris stay at
  the Institut Henri Poincar\'e.  Part of this material is based upon
  work supported by the National Science Foundation under Grant
  No. DMS-1440140 while the authors were in residence at the
  Mathematical Sciences Research Institute in Berkeley, California,
  during the Fall 2019 semester.  DB was supported by NSF CAREER grant
  DMS-1654056.  JW was supported by NSF grant DMS-1600023 and Simons
  Foundation Grant 631302.}

\begin{document}

\maketitle
\begin{abstract}
  The Dirac equation in $\RR^{1,3}$ with potential $\charge/r$ is a
  relativistic field equation modeling the hydrogen atom.  We analyze
  the singularity structure of the propagator for this equation,
  showing that the singularities of the Schwartz kernel of the propagator are along an
  expanding spherical wave away from rays that miss the potential
  singularity at the origin, but also may include an additional
  spherical wave of \emph{diffracted} singularities emanating from the
  origin.  This diffracted wavefront is $1-\ep$ derivatives smoother
  than the main singularities, for all $\ep>0,$ and is a conormal
  singularity.
  \end{abstract}

\section{Introduction}

In this paper we study the structure of the propagator for the
Dirac--Coulomb equation on $\RR^{1,3}$.  This equation, a description of the
hydrogen atom with a relativistic electron, was explicitly solved by
Darwin \cite{Da:28} in 1928 using separation of variables, giving a
mode-by-mode description of the solutions with the radial functions
defined by infinite series.  Such an approach, while computationally
useful for the spectral theory of the hydrogen atom, yields little
concrete information about the structure of the Schwartz kernel of the
propagator.

In this paper we derive the following results about the structure of
the propagator.  Notation involving the Dirac equation will be
explained in detail below.  Let $\eta$ denote the (mostly plus) Minkowski metric on
$\RR^4,$ whose coordinates are $t\equiv x^0,x^1,x^2, x^3.$  Let $r=r(x)\equiv\big(
(x^1)^2+(x^2)^2+(x^3)^2\big)^{1/2}$ denote radius in the space coordinates.

\begin{theorem}\label{theorem:structure} 
  Consider a real-valued vector potential $\pot = (A_{0} = \charge /
  r+V, A_{1}, A_{2}, A_{3})$ with $V, A_1,A_2,A_3 \in C^{\infty}(\RR^{3})$, and
let $m,\charge  \in \RR,$ with $\smallabs{\charge}<1/2.$

Let $\psi$ be the admissible fundamental solution of the Dirac equation minimally
coupled to the electric potential $V$:
  $$
\big( i \big(\gamma_0 (\pa_0+iA_{0} )+\gamma_j (\pa_j+i
A_j) \big) -m\big) \psi=0,
$$
with initial condition
$$
\psi_{x^0=0} =\psi_0 \delta_y
$$
for some four-spinor $\psi_0$ and point $y \in \RR^3.$

For $x^0 >r(y),$
$$
\WF \psi\subset \geom\cup \diff
$$
with $\geom=N^*\{\eta_{\alpha\beta} (x^\alpha-y^\alpha)
(x^\beta-y^\beta)=0\}$ given by the ``geometric'' (i.e., directly
propagated) light cone emanating from $y$ and
$\diff=N^*\{r(x)=x^0 -r(y)\}$ a secondary
``diffracted'' wavefront.  The singularity on $\diff \backslash \geom$
is conormal and is $1-0$ derivatives smoother than the singularity
at $\geom.$
\end{theorem}
(Here, as throughout the paper, we use the notation $a-0$ to mean
``$a-\epsilon$ for all $\epsilon>0$.'')  The notion of admissibility
of solutions, which simply refers to lying in the scale of energy
spaces defined by the self-adjoint Hamiltonian, is defined in
\S\ref{section:energy} below.
\begin{figure}\begin{center}
\begin{tikzpicture}[scale=2]
  \def\b{0.4};
  \def\a{-0.9};
  \def\l{1.5};
  \def\offset{\a-0.1}
  \shadedraw[left color=black, right color=white]   (\b-\l, \a+\l)  -- (\b, \a)  -- (\b+\l, \a+\l) --cycle;
\filldraw[color=white, draw=black]  (\b,\a+\l) ellipse ({\l} and {0.2*\l});
\shadedraw[opacity=0.5, left color=red, right color=white] (\b-\l,\a+\l)-- (0,\a+\b) -- (\l-\b,\a+\l) --cycle;
\filldraw[color=white, draw=black]  (0,\a+\l) ellipse ({\l-\b} and {0.2*(\l-\b)});
\draw[blue, thin,->] (-1.4, \offset) -- (1.8,\offset);
\draw[red, very thick,->] (0, -1.4) -- (0, 1.4);
\draw[blue, thin,<-] (-.9, -.9+\offset) -- (0.8, 0.8+\offset);
\node[align=left] at (\l-\b+0.1,\a+\l) {$\diff$};
\node[align=left] at (\b+\l+0.1,\a+\l) {$\geom$};
\filldraw (\b,\a) circle (0.02);
\node[align=left] at (\b+0.2,\a){$y$};
\node[align=left] at (0.2,1.4){$x^0$};
\node[align=left] at (2.0,\offset){$x^2$};
\node[align=left] at (-0.7,-0.9+\offset){$x^1$};
\end{tikzpicture}
\end{center}
\caption{
The ``geometric'' ($\geom$) and ``diffracted'' ($\diff$) wavefronts
for the fundamental solution with initial pole at $y.$  Note that the
main and diffracted fronts intersect along a single ray, the
continuation of the null geodesic from $y$ straight through the
potential singularity.}
\end{figure}
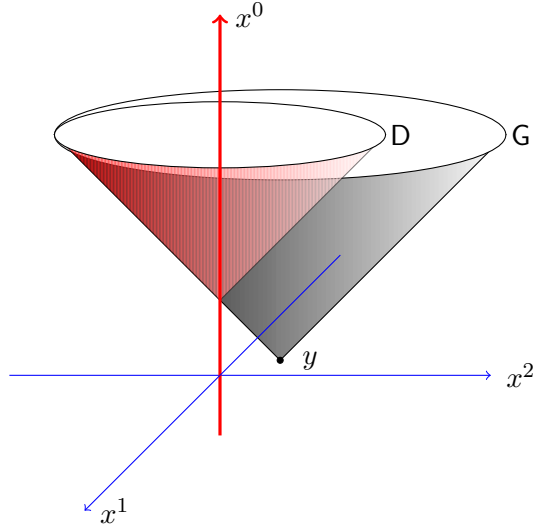


The proof uses tools originally developed for the
analysis of diffraction by cone and edge singularities
\cite{Melrose-Wunsch1}, \cite{MVW1}.  In particular, the analysis
proceeds in two main steps:
\begin{enumerate}
\item We show that the singularities of $\psi$ can at most lie in
  $\geom \cup \diff.$  This proceeds by a positive commutator argument
  using commutants in Melrose's b-calculus of pseudodifferential
  operators, inspired by the methods of Vasy \cite{Va:08}.
\item We show that the diffracted singularity is conormal and weaker
  than the main front.  This uses methods of Melrose and the second
  author from \cite{Melrose-Wunsch1}, involving Mazzeo's edge
  calculus of pseudodifferential operators, and a propagation of
  module regularity (as employed by Melrose--Vasy--Wunsch \cite{MVW1}))
  to obtain both the conormality and the
  regularity of the diffracted front.
  \end{enumerate}

  The Dirac--Coulomb equation describes spin-$\frac{1}{2}$ particles
  (such as electrons and positrons) in the presence of a point charge
  $\charge$.  Much of the literature about the Dirac--Coulomb system
  and related operators focuses on characterizing its eigenvalues and
  eigenstates.  This description is unfortunately insufficient to
  describe diffractive phenomena.  Darwin~\cite{Da:28} used separation
  of variables to characterize the generalized eigenfunctions of the
  exact Dirac--Coulomb system in terms of confluent hypergeometric
  functions and spinor spherical harmonics.  One could in principle
  derive our theorem in that setting by a careful analysis of the
  special functions but to our knowledge this has not been done.

  Kato in his book~\cite{Kato:book} provided one of the first results showing that
  the Hamiltonian governing the evolution of the Dirac--Coulomb system
  is essentially self adjoint in the range $\abs{\charge} < 1/2$
  (corresponding to atomic charge less than $68.5$).
  Weidmann~\cite{Weidmann} extended this result to $\abs{\charge} <
  \sqrt{3}/2$; beyond this value of $\charge$ the Hamiltonian is no
  longer self-adjoint.  We provide in Section~\ref{sec:self-adjoint}
  another proof of the essential self-adjointness in this optimal range.

  Other interest in the Dirac--Coulomb system as an evolution equation has
  come from the dispersive equations community.   Their work has largely
  focused on proving dispersive and Strichartz estimates for solutions
  by treating the components as solving systems of coupled wave
  equations.  We mention here the work of D'Ancona and
  collaborators~\cite{DaFa07, BoDaFa11, CaDa13} as well as the work of
  Cacciafesta--S\'er\'e~\cite{CaSe16} and Erdo{\u g}an--Green--Toprak~\cite{ErGrTo19}.

There is now a significant body of work describing the propagation
  of singularities on singular spaces, where diffraction occurs;
the problem of the wave equation
  on conic manifolds (or the wave equation with an inverse square
  potential) is the singular setting most closely resembling the
  Dirac--Coulomb problem.
  The
  first diffraction problems were rigorously analyzed by Sommerfeld
  \cite{Sommerfeld1}, with many other examples subsequently studied
  by Friedlander \cite{MR20:3703} and Keller \cite{Keller1}.  The use
  made by these authors of separation of variables and Bessel function
  analysis was generalized to cones of arbitrary cross section by
  Cheeger--Taylor~\cite{MR84h:35091a,MR84h:35091b}, who established the
  analogous result to Theorem~\ref{theorem:structure}
  in the setting of ``product cones,'' where the metric on the link
  does not vary with the radius.
  The non-product situation, where scaling invariance in $r$ is lost,
  requires different methods, and in consequence the
  b-pseudodifferential analysis used in this paper can be viewed as a
  continuation of a line of work beginning with
  Melrose--Sj{\"o}strand~\cite{Melrose-Sjostrand1,
    Melrose-Sjostrand2}, Melrose \cite{Melrose14}, and Taylor \cite{Taylor1} describing the propagation of singularities
  on manifolds with smooth boundary.  
 Melrose and
  the second author~\cite{Melrose-Wunsch1} used such commutator
  methods to generalize the results of Cheeger--Taylor to
  the non-product setting (see also Qian~\cite{Qian1} in
  the case of  inverse square potentials).  This work
  was expanded to include corners and edge singularities by Vasy~\cite{Va:08}
 and Melrose--Vasy--Wunsch~\cite{MeVaWu:13}, \cite{MVW1}.  The
  functional framework for our estimates is especially inspired by
  Vasy's work.

  One of the original applications for the careful analysis of
  singularity propagation was to the problem of wave decay.  Indeed,
  in certain settings Lax--Phillips~\cite{Lax-Phillips1} and
  Vainberg~\cite{Vainberg1} (later generalized by
  Tang--Zworski~\cite{Tang-Zworski1}) provided a blueprint for
  obtaining decay estimates on ``perturbations'' of odd-dimensional
  Euclidean spaces from propagation estimates using as input the weak
  Huygens principle, which dictates that a solution with compactly
  supported Cauchy data eventually
  becomes smooth in a fixed compact set.  
 More recent approaches to wave decay applying to spacetimes with
 ends that are not flat Minkowski space (again following the work of
  Vasy~\cite{Vasy:KdS}) give new ways to extract decay
  rates for solutions of wave equations from propagation estimates.
  Work of the authors and Vasy~\cite{BaVaWu:15, BaVaWu:18} and the
  first author and Marzuola~\cite{BaMa1} use related techniques to
  describe the radiation field on asymptotically Minkowski spaces and
  on product cones, respectively.  Similar techniques played a key
  role in the work of Hintz--Vasy~\cite{HiVa:18} establishing the
  global stability of the Kerr--de Sitter spacetime.

We thus hope to
  use the results obtained here to study the decay rates and
  asymptotics of the Dirac equation with one or more Coulomb-type singularities.
Additionally, there are potential applications of our results to
quantum field theory, viz., 
  the construction of Hadamard states for the Dirac--Coulomb
  problem (see, e.g., \cite{GeWr:14}). These physically acceptable states are
  characterized by their wavefront sets, with the separation between
  $\tau \gtrless 0$ components (with $\tau$ dual to $t$) playing an
  essential role.

  Even though the square of the Dirac--Coulomb system is principally
  scalar, the Dirac--Coulomb problem poses a number of difficulties not present
  with scalar wave equations on singular backgrounds.  Many of these
  can be described in terms of the form of the second order equation
  obtained by (approximately) squaring the system (described in
  Section~\ref{section:KG} below).  In the case of the exact
  Dirac--Coulomb system, this second order operator has the form
  \begin{equation*}
     -(\pa_t+ i \frac{\charge}{r})^{2} - \Lap -m^{2} - i\frac{\charge}{r^{2}}
     \begin{pmatrix}
       0 & \sigma_{r} \\ \sigma_{r} & 0
     \end{pmatrix},
  \end{equation*}
  where $\Lap$ is the (positive) Laplacian on $\reals^{3}$ and
  $\sigma_{r}$ are $2\times 2$ Pauli-type matrices that square to the
  identity.  The equation differs from the Klein--Gordon equation in two significant
  ways.  The first way is that the potential is coupled via the
  ``minimal coupling'' formalism, which introduces cross terms of the
  form $\frac{\charge}{r}D_{t}$; this does not present much additional
  difficulty, although it does need to be controlled in the b-calculus
  propagation arguments.  More significant is the second difference, namely the
  order zero term
  \begin{equation*}
    - i\frac{\charge}{r^{2}}
     \begin{pmatrix}
       0 & \sigma_{r} \\ \sigma_{r} & 0
     \end{pmatrix}.
  \end{equation*}
  As the Hardy inequality on $\reals^{3}$ suggests that factors of $1/r$
  should be treated as \emph{derivatives}, this term is principal from the
  point of view of scaling.  Moreover, it is anti-self-adjoint and
  cannot have a sign because $\sigma_{r}$ has eigenvalues $\pm 1$.  
  Dealing with it directly can cause significant headaches.  In trying to
  prove the diffractive theorem (Theorem~\ref{theorem:diffractive1}
  below) for the second order equation, this anti-self-adjoint term
  creates what should be viewed as the top order term and cannot be
  controlled by the positive terms in the commutator estimate.  This
  term even makes global energy estimates difficult, as the
  derivative of the energy can no longer be controlled by the energy.

  The complications of the Klein--Gordon system suggest that one ought to work with the first order
  system directly.  On the other hand, the ``energy estimates''
  obtained via the first order system are not as simple to work with
  as those arising from the second order equation.  We therefore use
  both equations in this paper.  For the elliptic part of the
  diffractive theorem (Section~\ref{sec:elliptic}) and the geometric
  improvement  (Section~\ref{sec:edgeprop}) we work with the second
  order equation, but for the ``hyperbolic'' part of the diffractive
  theorem (Section~\ref{sec:hyperbolic-diffractive}) we work directly
  with the first order equation.

  Studying the massive (rather than massless) Dirac equation
  introduces further complications.  In the massive case, the
  equations involving the $4\times 4$ Dirac matrices cannot be
  substantially simplified; in the massless ($m=0$) setting, the
  equations effectively decouple into two systems involving
  $2\times 2$ matrices.  More significantly, the presence of the mass
  term disrupts the commutation of the equation with the scaling
  vector field.  In the massless setting, it is possible to show that
  the diffracted wave has a leading order polyhomogeneous term but
  even this statement seems to be considerably more difficult in the
  massive case.

  In Section~\ref{sec:DC-intro} we introduce the Dirac--Coulomb
  equation and fix some notation.  Section~\ref{section:bedge}
  provides an introduction to the $\bo$- and edge-pseudodifferential
  calculi and describes the interaction of the $\bo$-calculus with
  differential operators on $\reals^{3}$.  In
  Section~\ref{sec:analyt-prlim} we return to the equation and provide
  some preliminary results: we show that the Hamiltonian governing the
  evolution is essentially self-adjoint for
  $\abs{\charge} < \sqrt{3}/2$, discuss the available energy
  estimates, introduce the second order operator, and describe how
  singularities propagate away from the origin.
  Sections~\ref{sec:diffractive} and~\ref{sec:geom} are the heart of
  the paper; Section~\ref{sec:diffractive} proves the diffractive
  theorem in which we show that singularities propagating through the
  origin must lie on the union of the diffracted and propagated fronts
  and Section~\ref{sec:geom} shows that the singularity along the
  diffracted front is $1-0$ orders smoother than along the propagated
  one.

  \section{The Dirac--Coulomb equation}
  \label{sec:DC-intro}
\subsection{Notation}
\label{sec:notation}
We use coordinates $x^\alpha,$ $\alpha=0,\dots, 3$ on $\RR^{1,3};$
when referring to spatial coordinates (indices $1,2,3$) we use Latin
rather than Greek superscripts.  When appropriate, we employ
the notation $t=x^0$ and use polar coordinates $r \in (0,\infty),$
$\theta \in S^2$ in the spatial variables.  Below and in what follows,
we use $\pot$ to denote an electromagnetic potential with $A_{\mu}$
its components, i.e., $\pot = (A_{0}, A_{1}, A_{2}, A_{3})$.  We are
most interested in the case when $A_{0}$ has Coulomb-like
singularities; in this case we write
\begin{equation*}
  A_{0}= \frac{\charge}{r} + V,
\end{equation*}
where $V \in \CI$.
  
The Dirac operator on $\RR^{1,3}$ is given by
$$
\dirac = \gamma^\alpha \pa_\alpha,
$$
where $\gamma^{\alpha}$ are the $4\times 4$ matrices
$$
\gamma^0=\begin{pmatrix} I & 0 \\ 0 & -I \end{pmatrix},
$$
and
$$
\gamma^j=\begin{pmatrix} 0 & \sigma_j \\ -\sigma_j & 0 \end{pmatrix},
$$
and $\sigma_j$ are the Pauli matrices,
$$
\sigma_1 =
\begin{pmatrix}
0 & 1 \\ 1 & 0
\end{pmatrix},\
\sigma_2 =
\begin{pmatrix}
0 & -i \\ i & 0
\end{pmatrix},\
\sigma_3 =
\begin{pmatrix}
1 & 0 \\ 0 & -1
\end{pmatrix}.
$$

The $\gamma$ matrices satisfy the anticommutation
relation\footnote{Readers consulting other references should be aware
  that there are at least two conventions in the literature.  Indeed,
  many physics texts (e.g., Akhiezer and Berestetsky~\cite{AkBe:65}
  and Rose~\cite{Ro:61}) ask that the gamma matrices satisfy a
  \emph{Riemannian} anticommutation relation and then set $x_{0} = ict$.}
\begin{equation*}
  \gamma^{\alpha} \gamma^{\beta} + \gamma ^{\beta}\gamma^{\alpha}= -2
  \eta^{\alpha\beta} \Id_{4},
\end{equation*}
where $\eta^{\alpha\beta}$ are the components of the Minkowski metric,
i.e.,
\begin{equation*}
  \eta^{\alpha\beta} =
  \begin{cases}
    -1 & \alpha = \beta = 0 \\
    1 & \alpha = \beta \in \{ 1, 2, 3\} \\
    0 & \alpha \neq \beta
  \end{cases}
  .
\end{equation*}

The free Dirac equation then reads
\begin{equation}\label{dirac}
(i \dirac-m)\psi=0.
\end{equation}
With an electromagnetic potential
$\pot = (A_{0}, A_{1}, A_{2}, A_{3})$, we replace $\dirac$ by
$$
\dirac_\pot \equiv \gamma^0 (\pa_0+iA_{0}) + \gamma^j (\pa_j + i A_{j});
$$
this is the ``minimal coupling'' convention.

Other notational conventions that we employ are as follows.  We use a
boldface Greek letter (such as $\bsigma$) to denote the associated
3-vector of matrices (such as $(\sigma_{1}, \sigma_{2}, \sigma_{3})$).
We then set
\begin{equation}
  \label{Sigma}
  \bSigma\equiv
  \begin{pmatrix}
    \bsigma & 0 \\ 0 & \bsigma
  \end{pmatrix}.
\end{equation}
and, in keeping with physics notation, we also write
$$
\beta =\gamma^0,
$$
and let $\balpha$ be defined by
$$
\bgamma=  \beta \balpha,
$$
hence
$$
\balpha=\begin{pmatrix} 0 & \bsigma \\ \bsigma & 0\end{pmatrix}.
$$
Letting
\begin{equation}
  \label{gamma5}
  \gamma_5 = i \gamma^{0}\gamma^{1}\gamma^{2}\gamma^{3} = \begin{pmatrix}
0 & \Id \\ \Id & 0
\end{pmatrix},
\end{equation}
we then obtain
$$
\balpha=\gamma_5 \Sigma.
$$

When using spherical coordinates, we will require radial versions of
various of the matrix quantities discussed above.  To this end, we set
\begin{equation}\label{radialmatrices}
\sigma_r=\sum_{j=1}^3 \frac{\hat x_j}{\smallabs{x}} \sigma_j,\  \alpha_r=\sum_{j=1}^3
\frac{\hat x_j}{\smallabs{x}} \alpha_j,\ \Sigma_r=\sum_{j=1}^3
\frac{\hat x_j}{\smallabs{x}} \Sigma_j,
  \end{equation}

\subsection{Spherical spinors and separation}
\label{sec:separation}
Let $$\bL=\br \times \bp,$$ where as usual
$$
\bp=
\begin{pmatrix}
  i^{-1} \pa_{x^1}\\
  i^{-1} \pa_{x^2}\\
    i^{-1} \pa_{x^3}
\end{pmatrix}.
$$
Let
$$
\bJ \equiv \bL +\frac 12 \bSigma
$$
denote the total angular momentum operators (orbital angular momentum
and spin together.)
Following Dirac, we also let
$$
K = \beta(1+\bSigma\cdot \bL).
$$
\begin{lemma}
Suppose $A_{0}$ is radial and $A_{j}= 0$.  The following operators are mutually commuting:
$$
\dirac_{\pot},\ J^2,\ J_{3},\ K. 
$$
Moreover,
$$
[\beta,K]=0.
$$
\end{lemma}
(See e.g.\ \cite{Ro:61}, Section 12 for proofs.)  In the case where
the potential $A_{0}$ is exactly radial, we could separate variables
explicitly and study the action of $\dirac_{\pot}$ on the common
eigenfunctions of the remaining operators.  Although we do not take
this approach, we include a discussion of the eigenfunctions because
some of the calculations below are easier to verify on individual
eigenspaces.  These eigenfunctions are well known to be described
blockwise by two component spinor spherical harmonics as follows.
Following e.g., \cite{Sz:07}, we set for $\theta \in S^2$
$$
\Omega_{\kappa \mu} (\theta) =
\begin{pmatrix}
  \sgn(-\kappa) \big( \frac{\kappa+1/2-\mu}{2 \kappa+1}
  \big)^{1/2}Y_{l,\mu-1/2}(\theta)\\
\big( \frac{\kappa+1/2+\mu}{2 \kappa+1} \big)^{1/2}Y_{l,\mu+1/2}(\theta)
\end{pmatrix},
$$
where
\begin{align}
\kappa &\in \ZZ \backslash \{0\},\\ \mu &\in
  \{-\smallabs{\kappa}+1/2,\dots, \smallabs{\kappa}-1/2\},\\
  l&=\abs{\kappa+\frac 12}- \frac 12,
\end{align}
and where $Y_{lm}$ are the standard spherical harmonics (see
\cite[(2.1.9)--(2.1.10)]{Sz:07} for normalization conventions).  Then
by \cite[(3.2.3)]{Sz:07}, we obtain
$$
(\bsigma \cdot \bL +1) \Omega_{\kappa\mu} = -\kappa\Omega_{\kappa\mu},
$$
hence
\begin{equation}\label{Kaction}
K
\begin{pmatrix}
 a \Omega_{\kappa\mu}\\ b \Omega_{\kappa'\mu'}
\end{pmatrix}= 
\begin{pmatrix}
   -a\kappa \Omega_{\kappa\mu}\\ b \kappa' \Omega_{\kappa'\mu'}
\end{pmatrix}
\end{equation}
and eigenvectors of $K$ are given by the span of
$$
\begin{pmatrix}
  \Omega_{\kappa\mu}\\ 0
\end{pmatrix},
\begin{pmatrix}
 0\\  \Omega_{-\kappa\mu'}
\end{pmatrix},\ \mu, \mu' \in   \{-\smallabs{\kappa}+1/2,\dots, \smallabs{\kappa}-1/2\};
$$
the eigenvalue of $K$ on this eigenspace is $-\kappa$.  Note that
\begin{equation}\label{sigmar}
\Sigma_r
\begin{pmatrix}
 a \Omega_{\kappa\mu}\\ b \Omega_{-\kappa\mu'}
\end{pmatrix}= 
\begin{pmatrix}
   -a \Omega_{-\kappa\mu}\\ -b\Omega_{\kappa\mu'}
\end{pmatrix},
\end{equation}
where $\Sigma_r$ is defined in \eqref{radialmatrices} above.

We further record here the relationship between $K$ and $\Lap_{\theta}$:
\begin{equation*}
  \Lap_{\theta} = K^{2} - \beta K.
\end{equation*}
This follows from the identity (see Rose~\cite{Ro:61}):
\begin{equation*}
  (\Sigma \cdot \bA) (\Sigma \cdot \bB) = \bA \cdot \bB + i \Sigma
  \cdot (\bA \times \bB).
\end{equation*}
Applying this to $\Sigma \cdot \bL$ yields
\begin{equation*}
  (\Sigma \cdot \bL) ^{2} = \Lap_{\theta} - \Sigma \cdot \bL,
\end{equation*}
so that
\begin{equation}\label{Ksquared}
  K^{2} = (\Sigma \cdot \bL + 1)^{2}  = \Lap_{\theta} + \Sigma \cdot
  \bL + 1 = \Lap_{\theta} + \beta K.
\end{equation}

In particular, note that
\begin{equation*}
  [\Lap_{\theta}, K] = [K^{2} - \beta K, K] = 0,
\end{equation*}
i.e., $K$ commutes with $\Lap_{\theta}$.

We now describe the separation of variables for a stationary Dirac
equation:
The massive Dirac equation with an electromagnetic potential $\pot =
  (A_{0}, A_{1}, A_{2}, A_{3})$ reads
$$
(i\dirac_{\pot}-m)\psi \equiv \big( i \big(\gamma^{0} (\pa_0+iA_{0}) +\gamma^{j} (\pa_j+iA_{j})\big) -m\big) \psi=0,
$$
hence multiplying by $\beta\equiv \gamma^0$ we obtain
$$
\big(  \big( i(\pa_0+iA_{0}) +i\beta \gamma^{j} (\pa_j+iA_{j})\big) -m\beta \big) \psi=0,
$$
i.e.,
\begin{equation}
 \dop\psi \equiv\big( i(\pa_0+iA_{0}) + i\alpha_j (\pa_j+iA_{j}\big) -m\beta \big)
\psi\equiv  i\pa_t\psi -\ellipticdirac \psi=0,
\label{eq:defin-of-first-order-op}
\end{equation}
where this is taken as a definition of the operator $\dop$ and 
$$
\ellipticdirac \equiv \sum_{j=1}^3 \alpha_j \frac{1}{i}(\pa_j+iA_{j}) +A_{0} + m\beta ;
$$
here we have, exceptionally, written out the summation explicitly here to remind the
reader that it is only over spatial indices $1,2,3.$  

Thus we are concerned with the unitary group generated by the
operator $\ellipticdirac.$

Now we compute, in the notation of \cite{Ro:61}, 
\begin{align*}
\balpha \cdot \frac{1}{i}\nabla
  &=\balpha \cdot \bp\\
  &= \gamma^{5}\bSigma\cdot \bp \\
  &= \gamma^{5}\bSigma_{r}\left(\frac{1}{i}\pa_{r} + \frac{i}{r}
    \bSigma \cdot \bL\right) \\
  &= -i \alpha_{r}\left( \pa_{r} - \frac{1}{r}(\beta K -\Id)\right) \\
  &= -i \alpha_{r}\left( \pa_{r} + \frac{1}{r} - \frac{1}{r}\beta K\right).
\end{align*}
More detail for the above calculation can be found in
Rose~\cite[p.~158, eq~(2.47)]{Ro:61}.  

Thus, finally, in polar coordinates,
\begin{equation}\label{ellipticdirac2}
  \ellipticdirac = \big( - i \alpha_{r} \left( \pa_{r} + \frac{1}{r} -
    \frac{1}{r}\beta K\right) + A_{0} + \sum \alpha_{j}A_{j} + m \beta\big).
\end{equation}


\section{$\bo$- and edge-geometry}\label{section:bedge}

Owing to the need to microlocalize solutions finely at the potential
singularity, it is natural to introduce a new space obtained by
\emph{blowup} from our Minkowski space.  In the simple case under
consideration here, the blowup amounts to substituting the space
$$
X \equiv [\RR^3; \{0\}] \equiv [0,\infty)_r \times S^2_\theta
$$
for the Euclidean space $\RR^3,$ with the \emph{blowdown map}
$$
\blowdown\colon X \to \RR^3
$$
being the polar coordinate map $(r,\theta) \to r\theta;$ this is a
diffeomorphism away from the boundary $r=0$ (which is referred to as
the \emph{front face} of the blowup).  We will use the same notation for the blowdown
map in the full Minkowski space, where we introduce polar coordinates
in spatial variables only, hence set
$$
M \equiv [\RR^{1,3}; \RR\times \{0\}] \equiv \RR_t \times X.
$$

Both $X$ and $M$ are manifolds with boundary.  (That they are
noncompact as well will play no essential role in our analysis, owing
to the local nature of the propagation of singularities.)  We will
need to consider two separate calculi of pseudodifferential operators
on $M,$ yielding microlocalizations of two different Lie algebras of
vector fields.  The first, Melrose's \emph{b-calculus}
\cite{Melrose:APS}, contains as first order operators the vector
fields tangent to the boundary of $M$.  The second, Mazzeo's
\emph{edge calculus} \cite{MR93d:58152}, contains instead the vector fields that are
\emph{tangent to the fibers of the blowdown map} as well as to the
boundary, hence in particular, we obtain $r \pa_t$ rather than $\pa_t$
in the latter calculus.  We describe the important features of these two
calculi below.

\subsection{$\bo$-calculus}

Full technical details on the $\bo$-calculus can be found in the book
of Melrose~\cite{Melrose:APS}; see also the introductory article by Grieser~\cite{Gr:01}.

The space of \emph{b-vector fields}, denoted $\mathcal{V}_{\bo}(M),$ is the
vector space of vector fields on $M$
tangent to $\pd M$;
they are spanned over $C^{\infty}(M)$ by the vector fields
$r\pd[r],$ $\pd[t],$ and $\pd[\theta]$. We note that $r\pd[r]$ is
well-defined, independent of choices of coordinates, modulo $r \mathcal{V}_{\bo}(M)$; one may call this the
{\em b-normal vector field} to the boundary. One easily verifies that $\mathcal{V}_{\bo}(M)$ forms a Lie
algebra. The set of b-differential operators, $\Diffb^{*}(M)$, is the
universal enveloping algebra of this Lie algebra:
it is the filtered algebra consisting of operators of the form
\begin{equation}\label{exampleboperator}
A=\sum_{\smallabs{\alpha}+j+k\leq m} a_{j,k,\alpha}( r,t,\theta) (r D_r)^jD_t^k
D_\theta^\alpha \in \Diffb^m(M)
\end{equation}
 (locally near $\pa M$) with the coefficients $a_{j,k,\alpha} \in \CI(M).$

The
b-pseudodifferential operators $\Psib^{*}(M)$ are the ``microlocalization''
of this Lie algebra, formally consisting of (properly supported) operators of the form
$$
b(r,t, \theta,r D_r,D_t,  D_\theta)
$$
with $b(r,t,\theta,\sigma,\tau, \eta)$ a Kohn-Nirenberg symbol.

The space $\mathcal{V}_{\bo}(M)$ is in fact the space of sections of a
smooth vector bundle over $M,$ the \emph{b-tangent bundle}, denoted
$\Tb M.$ The sections of this bundle are of course locally spanned by
the vector fields $r\pa_r,\pa_t, \pa_\theta.$ The dual bundle to $\Tb M$ is
denoted $\Tbstar M$ and has sections locally spanned over $\CI(M)$ by
the one-forms $dr/r, dt, d\theta.$

The symbols of operators in $\Psib^*(M)$ are thus Kohn-Nirenberg
symbols defined on $\Tbstar M.$ The principal symbol map, denoted
$\sigma_{\bo},$ maps the classical subalgebra of
$\Psib^{m}(M)$ to 
homogeneous functions of order $m$ on $\Tbstar M.$ In the particular
case of the subalgebra $\Diff_{\bo}^{m}(M),$ if $A$ is given by
\eqref{exampleboperator} we have
$$
\sigma_{\bo}(A)=\sum_{\smallabs{\alpha}+j+k= m} a_{j,k,\alpha}(r,t,\theta) \sigma^j\tau^k
\eta^\alpha
$$
where $\sigma,\tau, \eta$ are ``canonical'' fiber coordinates on $\Tbstar M$
defined by specifying that the canonical one-form be 
$$
\sigma \frac{dr}r+\tau dt + \eta \cdot \frac{d\theta}{r}.
$$
As homogeneous functions of a given order on $\RR^{n} \setminus 0$ can be
identified with smooth functions on $S^{n-1}$, we sometimes view
$\sigmab$ as a smooth function on $\Sbstar M$.

We also identify a subalgebra of $\Psib(M)$ that will be essential for
the commutator argument in Section~\ref{sec:diffractive}.
\begin{definition}
  \label{definition:invariant}
  We say $A \in \Psib^{m}(M)$ is \emph{invariant} if it is scalar and
  invariant under the action of $\SO (3)$ on functions, i.e., if $A$
  is scalar and $R^{-1}AR = A$ for all $R \in \SO (3)$, where the
  action of $\SO (3)$ on functions is simply $Rf(x) = f(R^{-1}x)$.
\end{definition}

Any scalar symbol invariant under the (lifted) action of $\SO(3)$
on $\Tbstar M$ may be quantized to an invariant operator.
\begin{lemma}\label{lemma:invariantcomm}
Invariant
operators commute with $\Lap_{\theta}$ and $K$.
\end{lemma}
\begin{proof}
Let $A$ be invariant.  For each $j \in \{1,2,3\},$ $[A,L_j]=0$ since
the flowout of $L_j$ is in $SO(3).$  Since $\Lap=\bL \cdot \bL$ and
$K=\beta(1+\Sigma\cdot \bL)$ (and $A$ is scalar) we obtain the desired
commutation.
  \end{proof}

\begin{remark}
  Although invariant operators commute with $\Lap_{\theta}$ and $K$, they
  \emph{do not commute} with the matrices $\sigma_{r}$ (defined in
  \eqref{radialmatrices}).  Because $\sigma_{r}$ is independent of
  $r$, though, the terms arising from commuting an invariant operator
  with $\sigma_{r}$ will be microsupported away from the
  characteristic set and so will be handled by the elliptic estimate
  in the course of the hyperbolic estimate of
  Section~\ref{sec:hyperbolic-diffractive} below.
\end{remark}

In addition to the principal symbol map, describing the leading order
behavior of elements of $\Psib^{*}(M)$ in terms of the filtration,
there is a second map that measures the leading order behavior of the
operators at the front face $r=0,$ and which, together with the
principal symbol, measures the obstruction to compactness of
$\bo$-operators.  We will refer to this notion below only in the
simple case of b-differential operators, where it is simple to
describe, and we will work in just spatial variables on $X$ rather
than in spacetime.  Then this extra symbol, which is operator-valued,
is simply the new operator obtained by freezing coefficients of powers
of b-vector fields at the boundary.  If $A$ is given by
$$
\sum_{\smallabs{\alpha}+j\leq m} a_{j,\alpha}( r,\theta) (r D_r)^j
D_\theta^\alpha
$$
 we
thus define the \emph{indicial operator}
$$
I(A) =\sum_{\smallabs{\alpha}+j\leq m} a_{j,\alpha}( 0,\theta) (r D_r)^j
D_\theta^\alpha.
$$
$I$ is a homomorphism.  Operators in the
range of $I$, which
in terms of $r$ are now simply polynomials in $(rD_r),$ are thus
further simplified by Mellin transform in $r,$ hence the same
information is contained in the 
\emph{indicial family}
$$
I(A,\sigma) =\sum_{\smallabs{\alpha}+j\leq m} a_{j,\alpha}( 0,\theta) \sigma^j
D_\theta^\alpha.
$$
The \emph{boundary spectrum} of $A$ is then defined as
$$
\specb(A) =\{ \sigma \in \CC\colon I(A,\sigma) \text{ is not
  invertible on } \CI(S^2)\}.
$$
This set plays an important role in establishing the mapping
properties of b-operators---see \cite[Chapter 5]{Melrose:APS}.  It
also is a key ingredient in the identification of the domain of the
essentially self-adjoint Hamiltonian in Section~\ref{sec:self-adjoint} below.

Let $L^2_\bo(M)$ denote the space of square integrable functions with respect to
the \emph{b-density}
$$
\frac{dr}r \, dt \, d\theta.
$$
Note in particular that this space differs from $L^2(M),$ which here
denotes the space
with the usual metric density, and in particular
$$
L^2(M)=r^{-3/2} L^2_{\bo}(M).
$$
When emphasizing the use of the metric density, we will in fact write
$$L^2_g(M)\equiv L^2(M)$$ for added clarity.
We let $\bH^m(M)$ denote the
Sobolev space of order $m$ \emph{relative to $L^2_\bo(M)$} corresponding to
the algebras $\Diffb^m(M)$ and $\Psib^m(M)$.  In other words, for
$m\geq 0$, fixing $A\in\Psib^m(M)$ elliptic, one has $w\in\bH^m(M)$ if
$w\in L^2_\bo(M)$ and $Aw\in L^2_\bo(M)$; this is independent of the
choice of the elliptic $A$.  For $m$ negative, the space is defined by
duality.  (For $m$ a positive integer, one can alternatively give a
characterization in terms of boundedness of elements of $\Diffb^m(M)$.)  Let
$\bH^{m,l}(M)=r^{l}\bH^m(M)$ denote the corresponding weighted
spaces.  We will also use all these notions on $X$ rather than $M,$
simply omitting the $t$ variable.  Sometimes it will be convenient to
use the Sobolev spaces defined with respect to the \emph{metric}
density rather than the b density we have used here, and to that end
we set (on either $M$ or $X$)
$$
H^{m}_{\bo, g} \equiv r^{-3/2} \bH^m.
$$

Associated to an operator $A \in \Psib^m(M)$ is its
\emph{microsupport},
$$
\WFb'(A)\subset\Sbstar M.
$$
This closed subset is the essential support of the total symbol,
just as in the usual pseudodifferential calculus, and obeys the usual
microlocality property
$$
\WFb'(AB) \subset \WFb'(A) \cap \WFb'(B).
$$
Conversely, there is a notion of \emph{b-ellipticity} at a point,
obtained from the invertibility of the principal symbol.  Note that global
ellipticity is not sufficient to make an operator Fredholm over a
compact set in $X$; additional
decay at $r=0$ is required to ensure that the remainder term in a parametrix
argument is compact.

While there is a notion of wavefront set (lying in $\Sbstar M$)
associated to the b-calculus, we will require a slight variant of this
wavefront set in our estimates, hence we postpone discussion of $\WFb$
until we have introduced differential-b-pseudodifferential operators.

\subsection{Edge Calculus}

Full technical details on the edge calculus can be found in
Mazzeo \cite{MR93d:58152}.

The space of \emph{edge-vector fields}, denoted $\mathcal{V}_{\eo}(M),$ is the
vector space of vector fields on $M$
tangent to $\pd M$ as well as to the fibers of the fibration
$\blowdown: M \to \RR^4;$
they are spanned over $C^{\infty}(M)$ by the vector fields
$r\pd[r],$ $r \pd[t],$ and $\pd[\theta]$.  Like the b vector fields, $\mathcal{V}_{\eo}(M)$ forms a Lie
algebra. The set of e-differential operators, $\Diffe^{*}(M)$, is the
universal enveloping algebra of this Lie algebra:
it is the filtered algebra consisting of operators of the form
\begin{equation}\label{exampleeoperator}
A=\sum_{\smallabs{\alpha}+j+k\leq m} a_{j,k,\alpha}( r,t,\theta) (r D_r)^j(rD_t)^k
D_\theta^\alpha \in \Diffe^m(M)
\end{equation}
 (locally near $\pa M$) with the coefficients $a_{j,k,\alpha} \in \CI(M).$

The
edge-pseudodifferential operators $\Psie^{*}(M)$ are the ``microlocalization''
of this Lie algebra, formally consisting of (properly supported) operators of the form
$$
b(r,t, \theta,r D_r,r D_t,  D_\theta)
$$
with $b(r,t,\theta,\xi,\tau, \eta)$ a Kohn-Nirenberg symbol.  The
(non-canonical) map
from total symbols to operators will be denote $\Op_b.$

For the commutator arguments below, we will require a doubly-filtered
version of the edge calculus, where we also track variable growth or
decay at $r=0.$ In particular, if we set
$$
\Psie^{m,l}(M)=r^{-l} \Psie^{m}(M),
$$
then this is a doubly filtered algebra.  We remark that
the operators
that are residual in the sense of both decay and regularity are
$$
\Psie^{-\infty,-\infty}(M);
$$
the reader is cautioned that different conventions exist in the
literature for the sign convention on the $l$ index.

The space $\mathcal{V}_{\eo}(M)$ is in fact the space of sections of a
smooth vector bundle over $M,$ the \emph{edge tangent bundle}, denoted
$\Te M.$ The sections of this bundle are locally spanned by
the vector fields $r\pa_r,r\pa_t, \pa_\theta.$ The dual bundle to $\Te M$ is
denoted $\Testar M$ and has sections locally spanned over $\CI(M)$ by
the one-forms $dr/r, dt/r, d\theta.$

The symbols of operators in $\Psib^*(M)$ are thus Kohn-Nirenberg
symbols defined on $\Testar M.$ The principal symbol map, denoted
$\sigma_{\eo},$ maps the classical subalgebra of
$\Psie^{m,l}(M)$ to 
$r^{-l}$ times homogeneous functions of order $m$ on $\Tbstar M.$ In the particular
case of the subalgebra $\Diff_{\eo}^{m,l}(M),$ if $A$ is given by
\eqref{exampleeoperator} we have
$$
\sigma_{\eo}(r^{l} A)=r^{l} \sum_{\smallabs{\alpha}+j+k= m} a_{j,k,\alpha}(r,t,\theta) \xi^j\lambda^k
\zeta^\alpha
$$
where $\xi,\lambda,\zeta$ are ``canonical'' fiber coordinates on $\Testar M$
defined by specifying that the canonical one-form be 
$$
\xi \frac{dr}r+\lambda \frac{dt}r + \zeta \cdot d\theta
$$

As before we let $L^2_\bo(M)$ denote the space of square integrable functions with respect to
the \emph{b-density}
$$
\frac{dr}r \, dt \, d\theta.
$$
We let $\He^m(M)$ denote the
Sobolev space of order $m$ \emph{relative to $L^2_\bo(M)$} corresponding to
the algebras $\Diffe^m(M)$ and $\Psie^m(M)$.  In other words, for
$m\geq 0$, fixing $A\in\Psie^m(M)$ elliptic, one has $w\in\He^m(M)$ if
$w\in L^2_\bo(M)$ and $Aw\in L^2_\bo(M)$; this is independent of the
choice of the elliptic $A$.  For $m$ negative, the space is defined by
duality.  (For $m$ a positive integer, one can alternatively give a
characterization in terms of $\Diffe^m(M)$.)  Let
$\He^{m,l}(M)=r^{l}\He^m(M)$ denote the corresponding weighted
spaces.

There is a notion of edge microsupport
$$
\WFe'(A) \subset \Sestar M,
$$
as well as of edge ellipticity satisfying the usual properties.

We recall also that associated to the calculus $\Psie^{*,*}(M)$ is
associated a notion of Sobolev wavefront set:
$\WFe^{m,l}(w)\subset \Sestar M$ is defined only for $w\in \He^{-\infty,l}$
(since $\Psie(M)$ is not commutative to leading order in
the decay index); the definition is then $\alpha\notin\WFe^{m,l}(w)$ if
there is $Q\in\Psie^{0,0}(M)$ elliptic at $\alpha$ such that $Qw\in
\He^{m,l}(M)$, or equivalently if there is $Q'\in\Psie^{m,l}(M)$
elliptic at $\alpha$ such that $Q'w\in L^2_{\bo}(M)$.
See \cite[Section 5]{Melrose-Wunsch1} for a fuller list of the
properties of the edge calculus and wavefront set.

\subsection{The differential-pseudodifferential $\bo$-calculus}

The crux of the proof of the diffractive theorem in
Section~\ref{sec:diffractive} below lies in understanding the
interaction between differential operators and the pseudodifferential
$\bo$-calculus.
A crucial ingredient below will be the \emph{Hardy inequality}
\begin{lemma}\label{lemma:Hardy}
If $u \in H^1(\RR^n)$ with $n \geq 3,$ then
$$
\frac{(n-2)^2}{4} \int \frac{\abs{u}^2}{r^2} \, dx \leq  \int
\abs{\nabla u }^2 \, dx.
$$
\end{lemma}
We will use this inequality in $\RR^3,$ where it reads
\begin{equation}\label{Hardyconstant}
\smallnorm{r^{-1} u} \leq 2\smallnorm{\pa_r u}.
\end{equation}

As the Dirac operator is not a $\bo$-operator, it is
convenient to measure regularity with respect to the classical Sobolev
space $H^{1},$ pulled back to $X.$
\begin{lemma}
The pullback $\blowdown^{*}(H^1)$ agrees with $\dom=r^{1}
H^1_{\bo,g} =r^{-1/2}\bH^1$
  locally near $r=0,$ and this pullback is injective.
  \end{lemma}
  \begin{proof}
We take all functions below to be supported in the unit ball.
    
The injectivity of the pushforward is assured by the fact that for all
$u \in H^1(\RR^3),$ if 
$\chi(r)$ is a cutoff function equal to $1$ for $r>2$ and $0$  for $r<1,$
the approximation $\chi(r/\ep)u$ converges to $u$ in $H^1(\RR^3)$
norm, i.e.\ elements supported away from the origin are dense in
$H^1,$ and it suffices to show that the pushforward is bounded above
and below as a Hilbert space map when acting on these distributions.
Since $\nabla u \sim (\pa_r u, r^{-1} \pa _\theta u),$ the $H^1$ norm
of $u$ is bounded by the $r H^1_{\bo,g}$ norm of $\blowdown^{*} u;$ the
Hardy inequality ensures that $\smallnorm{r^{-1}\blowdown_* u}_{L^2}$ is controlled by the $H^1$
norm of $u,$ which then shows that the $r H_{\bo,g}^1$ norm of
$\blowdown^{*} u$ is controlled by the $H^1$ norm of $u.$
\end{proof}

 In Section~\ref{sec:diffractive}, we let $H^1(M)$ be the closure in
 the $H^1(\RR^{1+3})$ norm (identified via the blowdown $\blowdown$)
 of $\mathcal{C}^\infty_c(M).$  The lemma above can be rephrased as
 the statement that
 \begin{align*}
   H^{1}(M) &= \blowdown^{*}H^{1}(\RR^{1+3}), \\
   H^{1}(X) &= \blowdown^{*}H^{1}(\RR^{3}).
 \end{align*}

In this paper we will only be dealing with functions compactly supported in a fixed (large) neighborhood of $x=0$, and we note that on such functions,
$$
\norm{D_t u}^2 + \norm{D_r u}^2 + \norm{r^{-1}\nabla_\theta  u}^2
$$
\emph{is equivalent to $\norm{u}_{H^1}^2.$}  We will use this equivalence heavily.

To facilitate the accounting of error terms in
Section~\ref{sec:diffractive}, we will use the terminology
\begin{equation*}
  A \in \Diff^{m}\!\Psib^{s}
\end{equation*}
if
\begin{equation*}
  A = \sum_{j+k \leq m} r^{-j}D_{r}^{k}A_{j,k}
\end{equation*}
with $A_{j,k} \in \Psib^{s}$.  (Cf.\ \cite[Definition
2.3]{Va:08}; here we allow powers of $r^{-1}$ in addition to
differentiations.)  For such operators, we write
\begin{equation*}
  \WFb' A = \cup _{j,k} \WFb' A_{j,k}.
\end{equation*}
Vasy~\cite{Va:08} made extensive use of these spaces of operators in the setting
of manifolds with corners; many of the results below have analogues in
that paper.

The following lemma from~\cite[Lemma 8.6]{MVW1} (cf.\
also~\cite[Lemma 2.8]{Va:08}) shows that $\Diff^{*}\!\Psib^{*}$ forms
an algebra.
\begin{lemma}
  \label{lemma:differentialcommutator}
  Let $A \in \Psib^{m}(M)$ and let $a = \sigmab(A)$.  Then
  \begin{equation*}
    [D_{r}, A] = B + CD_{r},
  \end{equation*}
  with
  \begin{align*}
    B \in \Psib^{m}(M), &\quad C \in \Psib^{m-1}(M), \\
    \sigmab(B) = \frac{1}{i}\pd[r]a, & \quad \sigmab(C) =
                                       \frac{1}{i}\pd[\sigma] a;
  \end{align*}
  moreover,
  \begin{equation*}
    [r^{-1}, A] = r^{-1}C_{R} = C_{L}r^{-1},
  \end{equation*}
  where $C_{\bullet} \in \Psib^{m-1}(M)$ with
  \begin{equation*}
    \sigmab (C_{\bullet}) = \frac{1}{i}\pd[\sigma]a.
  \end{equation*}
\end{lemma}

As we will measure $\bo$-regularity with respect to $H^{1}$, we also
need to know that $\Psib^{0}$ is bounded on this space.
\begin{lemma}
  \label{lemma:psib-on-h1}
  Given $A \in \Psib^{0}$, there is some $C > 0$ so that for all $u
  \in H^{\pm 1}$, 
  \begin{equation*}
    \norm{Au}_{H^{\pm 1}}\leq C \norm{u}_{H^{\pm 1}}.
  \end{equation*}
\end{lemma}

\begin{proof}
  We begin by proving boundedness on $H^1.$
  By Lemma~\ref{lemma:differentialcommutator}, $[D_{r}, A] = S +
  TD_{r}$, where $S \in \Psib^{0}$ and $T \in \Psib^{-1}$, so that
  \begin{align*}
    \norm{D_{r}Au}_{L^{2}_{g}} &\leq \norm{A D_{r}u}_{L^{2}_{g}} +
                             \norm{[D_{r}, A]u}_{L^{2}_{g}}   \\
                           &\leq \norm{AD_{r}u}_{L^{2}_{g}} +
                             \norm{Su}_{L^{2}_{g}} +
                             \norm{TD_{r}u}_{L^{2}_{g}} \\
                           &\leq C \left( \norm{D_{r}u}_{L^{2}_{g}} +
                             \norm{u}_{L^{2}_{g}}\right) \leq C \norm{u}_{H^{1}}.
  \end{align*}
  Similarly, we may use Lemma~\ref{lemma:differentialcommutator} to write
  \begin{equation*}
    \left[ \frac{1}{r}D_{\theta} , A \right] = \frac{1}{r} \left[
      D_{\theta} , A\right] + \left[ \frac{1}{r} , A\right] D_{\theta} =
      \frac{1}{r} S + T\left( \frac{1}{r}D_{\theta}\right),
  \end{equation*}
  where $S \in \Psib^{0}$ and $T \in \Psib^{-1}$, so that by the
  pseudodifferential calculus and the Hardy inequality we may bound
  \begin{equation*}
    \norm{\frac{1}{r}D_{\theta} A u}_{L^{2}_{g}} \leq C \norm{u}_{H^{1}}.
  \end{equation*}
The boundedness on $H^{-1}$ now follows by duality.
\end{proof}

The previous two lemmas then motivate a definition of $H^{1}$ (and $H^{-1}$)-based
$\bo$-wavefront set.
\begin{definition}
  Let $u \in H^{\pm 1}(M).$  Let $\rho \in \Tbstar M\backslash o.$  We define
  $$
  \rho \notin \WFb^{\pm 1,m} u
  $$
if there exists $A \in \Psib^m(M),$ elliptic at $\rho,$ such that $A
u\in H^{\pm 1}.$

Similarly, for $\rho \in \Tbstar M \backslash o$, we define
\begin{equation*}
  \rho \notin \WFb^{m}u
\end{equation*}
 if there exists $A\in \Psib^{m}(M)$, elliptic at $\rho$, such that
 $Au \in L^{2}_{g}$.
\end{definition}

\begin{remark}
  At this moment we provide the reader with two notes of caution:
  First, observe that we measure $\bo$-regularity with respect to
  $L^{2}_{g}$ rather than $L^{2}_{\bo}$; we adopt this convention
  because it makes applications of the Hardy inequality more
  straightforward and allows us to avoid introducing the weighted
  $\bo$-calculus.  Second, be aware that although $\WFb^{1,m}$ and
  $\WFe^{m,l}$ each have seem to have two superscripts, homologous
  indices have different meanings in these two objects.  Indeed, one
  should think of $\WFb^{1,m}$ as having only the index $m$ and
  therefore measuring $\Psib^{m}$-regularity \emph{with respect to
    $H^{1}$}.  On the other hand, $\WFe^{m,l}$ measures
  $\Psie^{m,l}$-regularity with respect to $L^{2}_{\bo}$ and thus has
  two indices corresponding to those of the edge algebra.
\end{remark}

As with other pseudodifferential algebras, it is convenient to know
that we can microlocalize our estimates:
\begin{lemma}
  \label{lemma:elliptic-estimation}
  If $A, G \in \Psib^{s}$ with $\WFb'A \subseteq \liptic G$, then for
  all $u$ with $$\WFb^{\pm 1,s}u \cap \WFb'G = \emptyset,$$
  we may bound
  \begin{equation*}
    \norm{Au}_{H^{\pm 1}}\leq C\left( \norm{Gu}_{H^{\pm 1}} +
      \norm{u}_{H^{\pm 1}}\right).
  \end{equation*}
\end{lemma}

\begin{proof}
  The proof is a standard microlocal elliptic parametrix argument: let $E \in
  \Psib^{-s}$ with $\WFb' E \subseteq \WFb' G$ so that
  \begin{equation*}
    R = \id - EG \in \Psib^{0}, \quad \WFb' R \cap \WFb' A = \emptyset.
  \end{equation*}
  We may then write
  \begin{equation*}
    Au = A (EG + R)u, 
  \end{equation*}
  so that
  \begin{equation*}
    \norm{Au}_{H^{\pm 1}} \leq \norm{(AE)Gu}_{H^{\pm 1}} +
    \norm{ARu}_{H^{\pm 1}}
    \leq C \left( \norm{Gu}_{H^{\pm 1}} + \norm{u}_{H^{\pm 1}} \right).
  \end{equation*}
\end{proof}

In Section~\ref{sec:diffractive}, we repeatedly use the algebra properties of $\Diff^{*}\!\Psib^{*}$ and the
following lemma to allow easy estimates on error terms by doing
commutations freely.
\begin{lemma}
  \label{lemma:order-shifting}
  Suppose $E \in \Diff^{1}\!\Psib^{s+r-1}+\Psib^{s+r}$.  There are pseudodifferential
  operators $A \in \Psib^{s-1}$ and $B \in \Psib^{r}$ with $\WFb'A
  \cup \WFb' B \subseteq \WFb'E$ so that for all $u \in H^{1}$ and
  $v\in L^{2}$ with $\WFb'E \cap (\WFb^{1,s-1} u \cup \WFb^{1, r-1}v)
  = \emptyset$, 
  \begin{equation*}
    \abs{\ang{E u, v}} \leq C\left( \norm{Au}_{H^{1}}\norm{Bv}_{L^{2}_{g}}
      +  \norm{u}_{H^{1}}\norm{v}_{L^{2}_{g}} \right).
  \end{equation*}

  Similarly, if
  $E \in \Diff^{2}\!\Psib^{s+r-2} +
  \Diff^{1}\!\Psib^{s+r-1}+\Psib^{s+r}$, we may find
  $A \in \Psib^{s-1}$ and $B\in \Psib^{r-1}$ so that
  \begin{equation*}
    \abs{\ang{E u, v}} \leq C\left( \norm{Au}_{H^{1}}\norm{Bv}_{H^{1}}
      + \norm{u}_{H^{1}}\norm{v}_{H^{1}} \right).
  \end{equation*}
\end{lemma}

\begin{proof}
  Let $T_{\pm r} \in \Psib^{\pm r}$ be elliptic, self-adjoint $\bo$-operators
  which are inverses of one another modulo a smoothing error, so that
  $T_{r}T_{-r} = \Id + R$ with $R \in \Psib^{-\infty}$.  Setting $A =
  T_{-r}E $ and $B = T_{r} \in \Psib^{r}$ finishes
  the proof.  
\end{proof}

\section{Analytic preliminaries}
\label{sec:analyt-prlim}

We return to the Dirac--Coulomb equation $(i \dirac_{\pot} -m)u = 0$.  In
this section we discuss several preliminary results needed in the main
proofs below.

\subsection{Self-adjoint extension}
\label{sec:self-adjoint}
Recall that  $$M=[\RR^{1+3}; \RR_t \times 0_x]$$ denotes the blowup of our
spacetime at the spatial origin,
and
$$
X=[\RR^{3}; \times 0]
$$
denotes its spatial cross section,
with $\blowdown$ denoting the blowdown map in either case. 

We shall abuse notation later on in
confusing $X$ with all of $\RR^3,$ but will begin by distinguishing
these two spaces for the purposes of describing domains and Sobolev
spaces precisely before proving that the confusion is safe.
  
We now examine the indicial roots of the formally self-adjoint
operator $\ellipticdirac$ (defined in Section~\ref{sec:notation}) where\footnote{More
generally, we remark that we can replace the smooth term by a term that is smooth on
the blowup of the origin with no change in the arguments of this section.} $A_{0}=\charge/r+V$ and
$V, A_{j}\in \CI,$ i.e., the boundary spectrum given by the points
of non-invertibility of $I(r\ellipticdirac, \xi).$  

By \eqref{ellipticdirac2}, if $\sigma$ denotes the dual to $r D_r$ in $\Tbstar X,$
$$
I (r\ellipticdirac,\xi)  =\charge \Id -i\gamma^5\Sigma_r\big( i \sigma  + 
1-\beta K\big).
$$

To study the equation $I(r\ellipticdirac,\xi) \psi=0$ we split
$$
\psi  =
\begin{pmatrix}
  \psi^u\\ \psi^l
\end{pmatrix}
$$
into upper and lower spinors, and, as above, expand each in the basis
of spherical spinors of the form
$$
\begin{pmatrix}
 \Omega_{\kappa\mu}\\
0
\end{pmatrix},\ \begin{pmatrix}
0\\
 \Omega_{-\kappa\mu'}
\end{pmatrix}.
$$

Thus, once again using \eqref{Kaction}, \eqref{sigmar},
 we obtain
$$
I (r\ellipticdirac,\xi)
\begin{pmatrix}
   a \Omega_{\kappa\mu}\\  b \Omega_{\kappa'\mu'}
\end{pmatrix} = 
\begin{pmatrix}
a\charge \Omega_{\kappa\mu} -(\sigma - i + i \kappa )b\Omega_{\kappa' \mu'}\\  
 b \charge \Omega_{\kappa'\mu'} -(\sigma - i - i \kappa ) a \Omega_{-\kappa \mu}
\end{pmatrix}.
$$
Hence there is only nullspace when 
$$
\charge^2= \kappa^{2} + ( \sigma - i )^{2},
$$
i.e.\ when 
$$
\sigma=i \pm i\sqrt{\kappa^{2}- \charge^{2}}.
$$

Because $\kappa$ takes values in $\ZZ \setminus \{ 0 \}$, we can
explicitly calculate these indicial roots for small values of $\charge$.
Indeed, if $\smallabs{\charge} < \sqrt{3}/2$, we are assured that
\begin{equation}\label{noindicialroots}
\Im \specb(r\ellipticdirac) \cap [1/2,3/2]=\emptyset.
  \end{equation}

Now since\footnote{Recall that $\bo$-Sobolev spaces are by default defined with respect to
the b-density rather than the metric density.}
$$
\ellipticdirac: r^{-1/2}\bH^1(X) \to r^{-3/2}L^2_b(X)=L^2_g(X)
$$
is continuous, we certainly find that $r^{-1/2}\bH^1(X)$ is contained
in the minimal domain of $\ellipticdirac.$
 On the other hand,
\eqref{noindicialroots} implies by work of Lesch
\cite[Corollary 1.3.17]{Le:97} (see also Melrose \cite[Chapter 5]{Melrose:APS}
for a parametrix construction, as well as Gil--Mendoza \cite{GiMe:03} for a general discussion of
self-adjoint extensions of operators of this type) that the maximal and minimal
domains must in fact coincide, hence
$\ellipticdirac$ is essentially self-adjoint, with domain
given by
\begin{equation}\label{domain}
\dom = r^{-1/2}\bH^1.
\end{equation}
(Cf.\ \cite[Theorem V.5.10, Remark
V.5.12]{Ka:66} for the essential self-adjointness of Dirac operators.)

Having established the self-adjointness of
  $\ellipticdirac$ with domain $\dom,$ we now define
  $$
\dom^s =\operatorname{Dom}  (\Id +\ellipticdirac^2)^{s/2},
$$
with the powers of the operator being defined by the spectral
theorem.  Note that away from the origin, these simply agree with
Sobolev spaces:
\begin{lemma}\label{lemma:domsob}
For all $s \in \RR$
  $$
\dom^s \cap \mathcal{E}' (\RR^3 \backslash \{0\})=H^s \cap \mathcal{E}' (\RR^3 \backslash \{0\}).
$$
\end{lemma}
\begin{proof}
For $s$ an even integer, the result follows inductively from the characterization of
$\dom=\dom^1,$ which does agree with $H^1$ away from the origin.
Thus for any $\varphi \in \mathcal{C}_c^\infty(\RR^3 \backslash
\{0\}),$ whenever $\Re s\in 2 \NN,$
\begin{equation}\label{samedomain}
  \varphi u \in \dom^s \Longleftrightarrow (\Id +\Lap)^s \varphi u \in L^2,
\end{equation}
since the pure imaginary powers of $(\Id+\ellipticdirac^2)$ and of
$(\Id+\Lap^2)$ are both unitary.  Then by interpolation and duality
\eqref{samedomain} holds for all $s.$
  \end{proof}

\subsection{Admissible solutions and energy estimates}\label{section:energy}
\begin{definition}
  A solution to $(i \dirac_v-m)u=0$ is \emph{admissible} if it lies in
  $$
\mathcal{C}(\RR; \dom^s)
$$
for some $s \in \RR.$
  \end{definition}
  \emph{In the propagation theorems in this paper, we deal only with
    admissible solutions.}  Note that there is a unique admissible fundamental solution,
 since the initial data $\delta(x-x_0)$ lies in
  $\dom^{-n/2-0}$ by Lemma~\ref{lemma:domsob}.

  Given Cauchy data $u_0 \in \dom^s,$ there exists a unique admissible
  solution
  $$
e^{-it \ellipticdirac} u_0
  $$
by Stone's theorem; the propagator is of course unitary on $\dom^s$
for all $s \in \RR.$
More generally, we will have use for the following energy estimate:

\begin{lemma}\label{lemma:diracenergy}
  Let $u$ solve $(i \dirac_V-m)u=0$ on $[t_0, t_1] \times X$ and lie in $\CI(\RR; \dom^{\infty}).$    For any operator $Q: \CI(\RR; \dom^\infty) \to \CI(\RR; \dom^\infty),$
    $$
\frac 12 \frac{d}{dt} \norm{Q u}_{\dom^s}^2 =\Re \ang{i [\ellipticdirac,Q] u, Qu}_{\dom^s}
    $$
  \end{lemma}
  \begin{proof}
This follows by self-adjointness of $\ellipticdirac$ and the
definition of the $\dom^s$ norm in terms of its powers.
    \end{proof}

    For purposes of shifting regularity of solutions up and down
    conveniently, we now define, for $s \in \RR,$
    $\Theta_s\in \Psi^s(\RR)$ to be a parametrix for $\ang{D_t}^s$
    whose Schwartz kernel is properly supported; thus
    $\Theta_s \Theta_{-s}-\Id$ is a smoothing operator with properly
    supported Schwartz kernel.  We then note by $t$-translation
    invariance of the Dirac equation that if
    $u \in \mathcal{C}(\RR; \dom^k)$ is a solution to the Dirac
    equation, then (by ellipticity of the spatial part of the Dirac
    operator)
$$
\Theta_s u \in \mathcal{C}(\RR; \dom^{k-s})\cap H^{k-s}_{\loc}
$$
is another solution (up to a smooth remainder), and
$$
\Theta_{-s} \Theta_s u -u \in 
\CI(\RR; \dom^\infty).
$$

It is helpful in what follows to be able to pass freely among
different notions of solution: viewing a solution as lying in
locally $H^s(\RR\times \RR^3)$ is most natural in dealing with
microlocal analysis away from $r=0,$ while the energy spaces
$L^2(\RR;\dom^s)$ or $\mathcal{C}(\RR;\dom^s)$ are natural from the
point of view of global energy estimates.
\begin{lemma}
An admissible solution of the Dirac equation in $\mathcal{C}(\RR;
\dom^s)$ lies in $H^s_{\loc}(M^\circ).$
\end{lemma}
  \begin{proof}
    For such a solution (with all norms below local ones, for $t$ in a
    finite interval, and over a compact set in the interior of $X$)
    $$
\Theta_s u \in \mathcal{C}(\RR; L^2) \subset L^2,
$$
hence
$$
u \in H^s(\RR; L^2) \cap L^2(\RR; H^s) \subset H^s,
$$
by the local Fourier characterization of Sobolev regularity (and since
$\smallabs{\tau}^s +\smallabs{\zeta}^s \sim \smallabs{(\tau,\zeta)}^s$
outside the unit ball).
  \end{proof}

\subsection{Reduction to Klein--Gordon}\label{section:KG}
Some of the arguments below are considerably simplified by considering
a related principally scalar second-order operator obtained essentially by squaring the
Dirac operator.  

Consider a four-spinor solution $u$ to 
$$
(i \dirac_\pot-m)u=0,
$$
where $\pot=(A_{0}, A_{1}, A_{2}, A_{3})$, i.e.,
$$
\big( i(\gamma^0 (\pa_0+iA_{0}) +\gamma^j (\pa_j+i\pot_{j}))-m \big) u=0.
$$

Applying $(i \dirac_\pot +m)$ we obtain immediately
$$
\begin{aligned}
  0 &= (-\dirac_\pot^2-m^2) u\\
  &= - \big( \gamma^0 (\pa_0+iA_{0}) + \gamma^j(\pa_j+iA_{j})\big) \big( \gamma^0
  (\pa_0+iA_{0}) + \gamma^k(\pa_k+iA_{k})\big) u-m^2u\\
  &= -\big( (\pa_0+iA_{0})^2u -(\pa_j+iA_{j})^2 u + \gamma^j \gamma^0 (i
  \pa_j (A_{0}))u + \gamma^{j}\gamma^{k}(i\pa_{j}A_{k}) u \big) -m^2 u\\
  &= - (\pa_0+iA_{0})^2u +(\pa_j+iA_{j})^2 u-m^2u  -i \gamma^j \gamma_0 
  \pa_j (A_{0}) u  -i\gamma^{j}\gamma^{k}\pa_{k}(A_{k})u - m^2 u .
  \end{aligned}
$$
For $A_{0}$ radial,
$$
\begin{aligned}
  -i \gamma^j \gamma^0 \pa_j(A_{0})  &= i\gamma^0 \gamma^j \pa_j(A_{0})\\
  &= i\gamma^0 \gamma_r \pa_r(A_{0})\\
  &= i \begin{pmatrix} I & 0 \\ 0 & -I \end{pmatrix} \begin{pmatrix} 0 &
    \sigma_r\\ -\sigma_r & 0 \end{pmatrix} \pa_r (A_{0})\\
  &=i \begin{pmatrix} 0 &
    \sigma_r\\ \sigma_r & 0 \end{pmatrix} \pa_r (A_{0}),
\end{aligned}
$$
hence, for $A_{0}$ radial and $A_{j} = 0$,
$$
- (\pa_0+iA_{0})^2u +\pa_j^2 u-m^2u+ i \begin{pmatrix} 0 &
    \sigma_r\\ \sigma_r & 0 \end{pmatrix} \pa_r (A_{0})u=0.
  $$

  More generally, assume $A_{j}\in \CI$ and
  \begin{equation*}
    A_{0} = \frac{\charge}{r} + V,
  \end{equation*}
  where $V\in \CI$.  We now lump the extra terms together as
  perturbations, and multiply through by $\gamma^0$ rewrite the first
  order equation in a more convenient form as 
\begin{equation}\label{eth}
  \dop  \equiv i (\pd[t] + i\frac{\charge}{r} + iV) + i
  \alpha_{r}\left( \pd[r] + \frac{1}{r} - \frac{1}{r}\beta K\right) -
  \sum_{j=1}^{3}\alpha_{j}A_{j} - m \beta,
\end{equation}
where we recall that $\beta = \gamma^{0}$ and $\alpha_{j} =
\begin{pmatrix}
  0 & \sigma_{j} \\ \sigma_{j} & 0 
\end{pmatrix}
$.  The corresponding operator of Klein--Gordon type
\begin{equation}\label{Pdef}
  P\equiv (i \dirac_\pot+m)(i \dirac_\pot-m)
\end{equation}
 then satisfies the
following hypotheses:
\begin{assumption*}
 $P$ is a second-order operator of
  the following form:
\begin{equation}\label{Pform}
  P= -(\pa_0+i\frac{\charge}{r})^2 + \sum \pa_j^2 -m^2
  -i\frac{\charge}{r^2} \begin{pmatrix} 0 &
    \sigma_r\\ \sigma_r & 0 \end{pmatrix}
+\bR
\end{equation}
with
\begin{equation}\label{Rform}
\bR=\charge \frac{\mathbf{W}_0}{r}+ \mathbf{W}_1^\alpha \pa_\alpha+\mathbf{W}_2
  \end{equation}
where $\mathbf{W}_\bullet\in \CI(\RR^3)$ but are not
necessarily scalar. 
\end{assumption*}
These assumptions on the operator will suffice for most of our
propagation results below.

Note that the term $$  -i\frac{\charge}{r^2} \begin{pmatrix} 0 &
    \sigma_r\\ \sigma_r & 0 \end{pmatrix}$$ is, in contrast to the
  other main terms in the equation, formally anti-self-adjoint rather
  than self-adjoint.  This creates significant technical difficulties in
  the b-propagation arguments, since, while lower order in terms of
  differentiation, this anti-self-adjoint term is large.  If we
  estimate it in pairings by the Hardy inequality, is larger
    than the second-order terms in the equation.  This obstacle is why
    we use the first order equation directly in the hyperbolic part of
    the b-propagation argument below.

The presence of the charge parameter multiplying
$\mathbf{W}_0$ is in fact inessential here, as its size will play no
role in the analysis of that term.

\subsection{Interior propagation}
\label{sec:interior-coiso}

In this section, we discuss the propagation \emph{away from $r=0$} of
singularities (or, dually, of regularity) and also of 
iterated regularity under the angular test operators $D_\theta$ as
well as the spacetime scaling vector field
\begin{equation}\label{Rdef}
  R=rD_r+t D_t.\end{equation}

First, we remark that away from the potential singularity at the
origin, the standard theory of propagation of singularities applies:
\begin{proposition}\label{prop:propsing}
Let $u$ satisfy $(i \dirac_\pot-m)u=0.$  Then
$\WF u \subset \Sigma\equiv \{ \eta^{\alpha \beta} \xi_\alpha \xi_\beta=0\}$ and is a union of
maximally extended integral curves of the Hamilton flow generated by
$\eta^{\alpha \beta} \xi^\alpha \xi^\beta,$ i.e., lifts of straight lines.
\end{proposition}
Here (and here alone) we have used $\xi_\alpha$ to denote the dual cotangent
variable to the Minkowski coordinate $x^\alpha.$
\begin{proof}
Applying $(i \dirac_\pot+m)$ yields $Pu=0.$  Since $P$ is an operator
of real principal type (away from the potential singularity), the
result follows from the theorem of H\"ormander \cite{Duistermaat-Hormander1}.
  \end{proof}

Now we turn to propagation of iterated regularity under $R, D_\theta.$
Note that this is a simple case of propagation of
\emph{test module regularity}, with $D_\theta$ together with $P$ being generators of a
module of operators testing for \emph{coisotropic} regularity relative
to the manifold $$\coiso \equiv \{\tau^2=\xi^2,\ \eta=0 \}\subset
T^*M^\circ$$ (using coordinates $\tau, \xi,\eta$ dual to $t,r,\theta$
respectively)  and with $D_\theta, R, P$
together testing for regularity relative to the \emph{Lagrangian}
submanifold(s)
$$
\lag\equiv N^* \{t=\pm r\}\subset \coiso
$$

\begin{proposition}\label{proposition:interiorcoiso}
  Let $u$ be an admissible solution to the Dirac equation.
  
  Let $p_0 \in \{\sigma(P)=0\} \subset T^* (M^\circ)$  and let
$p_1$ lie on the maximally extended null bicharacteristic through
$p_0$ in $T^*M^\circ.$ 

If $p_0 \notin \WF^s
D_\theta^\alpha u$ for all $\smallabs{\alpha}\leq N$ 
then $p_1 \notin \WF^s
D_\theta^\alpha u$ for all $\smallabs{\alpha}\leq N.$ 

Likewise, if $p_0 \notin \WF^s
R^j D_\theta^\alpha u$ for all $j+\smallabs{\alpha} \leq N$ then
$p_1 \notin \WF^s
R^j D_\theta^\alpha u$ for all $j+\smallabs{\alpha} \leq N.$
  \end{proposition}
  \begin{proof}
    The proof is a standard exercise in propagation of
   ``test module regularity'' and is essentially an easier version of
   the b- and edge-calculus arguments employed below to obtain
   propagation through the potential singularity, hence we merely
   sketch it (cf.\ \cite[Proposition~6.11]{MVW1}).

By Taylor's theorem and the symbol calculus, for a solution to $P u
\in \CI,$ the regularity hypothesis
   $p_0 \notin \WF^s
   D_\theta^\alpha u$ for all $\smallabs{\alpha}\leq N$ is
   microlocally equivalent to the assertion that for \emph{any} $A_1, \dots
   A_N \in \Psi^1(M^\circ)$ with proper support, characteristic on $\coiso,$
   $$
A_1\dots A_N u \in H^s.
$$
Now by \cite[Theorem 21.2.4]{Hormander:vol3}, we may find a homogeneous
symplectomorphism $\Phi,$ defined on a neighborhood of $p_0,$  mapping from coordinates $(y,z,\eta,\zeta)$
such that $\sigma(P)\circ \Phi =
\zeta_1 q$ with $q$ elliptic and $\Phi^{-1}(\coiso)=\{\zeta=0\}.$  We
may also assume $\Phi(p_0)=0$ and hence $\Phi(p_1)$ lies on the $z_1$-axis.

We
may then quantize $\Phi$ to a microlocally unitary FIO $T$ such that
$TP = QD_{z_1} T+E$ with $E \in \Psi^{-\infty},$ and where $Q \in
\Psi^1$ is elliptic.
Then $Pu=0$ implies $Q D_{z_1} Tu \in \CI,$ hence $D_{z_1} T u \in
\CI$ by ellipticity. The hypotheses are
equivalent to
$D_z^\alpha T u \in H^s$ near $\Phi(p_0)$ for all
$\smallabs{\alpha}\leq N.$ Solving the equation
$D_{z_1} Tu \in \CI$ then guarantees that the same holds near any
$\Phi(p_1)$ along the $z_1$-axis.

The second part of the result, dealing with Lagrangian regularity,
follows via the same kind of proof: here we conjugate instead to a
coordinate system $(z,\zeta)$ so that the
operators $P, R, D_\theta,$  whose symbols cut out the Lagrangian
$\lag,$ become multiples of the model operators $D_{z_j}$ and proceed
as before.
    \end{proof}

\section{Diffractive theorem}
\label{sec:diffractive}

\subsection{Main theorem}
In this section, we prove the \emph{diffractive theorem}, which tells
us that the only wavefront set emanating from the singularity of the
potential arises at the time of interaction with a singularity of the
solution.

In making our propagation arguments in the b-calculus we will study
  the Dirac equation directly.  It turns out
  to be simplest to deal with the Klein--Gordon operator $P,$ however,
  in making the \emph{elliptic} estimates that constrain where the
  b-wavefront set may lie.  We thus employ both equations in turn in
  proving the diffractive theorem.
  
In both settings, we deal with large potential terms
by employing the Hardy inequality, with the result that our
results only hold for $\abs{\charge}< 1/2$.

\begin{definition}
  A \emph{diffractive geodesic} is a geodesic that is either
  \begin{enumerate}
  \item a lightlike geodesic not passing through $r=0$, or
  \item a continuous concatenation of two lightlike geodesics, both
    passing through $t=t_0,\ r=0$ for some $t_0 \in \RR,$ hence in
    polar coordinates a geodesic passing through the origin at time
    $t=t_0$ with
    $$
r=\smallabs{t-t_0},\ \theta=\begin{cases}\theta_-, & t<t_0\\ \theta_+,& t>t_0.\end{cases}
    $$
    \end{enumerate}
\end{definition}
(\emph{Geodesic} here refers to a geodesic with respect to the
Minkowski metric, hence a straight line.) 
Note  that in the latter case, when the geodesic is broken, there is no
need for the arriving and departing \emph{spatial directions} of the geodesic
to match up as it enters and leaves the origin, though the direction
in time must be conserved.

 We will abuse
notation by using the term \emph{geodesic} interchangeably for the
curve in $M^\circ$ and for its lift to $T^* M^\circ.$

A simple version of the diffractive propagation theorem, making no
reference to b-wavefront set, says that the wavefront set of a
solution to the Dirac equation is, away from the spatial origin, given
by a union of lifts of diffractive geodesics to $T^*\RR^3.$  To prove
the theorem, however, requires proving uniform estimates at the time
the geodesic reaches $r=0,$ which requires analysis of the 
b-wavefront set; Proposition~\ref{prop:propsing} takes care of
propagation away from $r=0.$

In order to describe wavefront sets conveniently, we will use
coordinates associated to the canonical one-form
\begin{equation}\label{bform}
\sigma \frac{dr}r + \eta \cdot d\theta +\tau \, dt
\end{equation}
on $\Tbstar M.$  We may canonically identify this
cotangent bundle with $T^*\RR^4$ away from $r=0:$ this
follows from the observation that $\blowdown$ is a diffeomorphism
away from $r=0$ (identifying $T^{*}\RR^{4}$and $T^{*}M$ there) and
that the natural map $T^{*}M\to \Tbstar M$ is an isomorphism in this
region.

In the coordinates given by \eqref{bform}, for
the radial geodesics (i.e., integral curves of the Hamilton flow of
the metric), $dr/dt=-\sigma/r\tau,$ hence the set where
$\sigma$ and $\tau$ have the same sign should be viewed as
``incoming'' toward $r=0$ under the bicharacteristic flow, while the
set where they have opposite signs is outgoing.  Thus the following
theorem describes propagation into and then back out of the singular
point of the Coulomb potential.
\begin{theorem}\label{theorem:diffractive1}
  Let $\pot = (A_{0} = \charge/ r + V, A_{1}, A_{2},A_{3})$ with
  $V, A_{j} \in \CI(\reals^{3})$, and $\smallabs{\charge}<1/2.$

  Whenever $u$ is an admissible solution of
  $$
  (i\dirac_\pot-m) u=0,
  $$
  if 
  $$
  \{ (r=t_0-t, \theta,t, \sigma , \tau,\eta=0) \colon t<t_0,
  \theta \in S^2, \sigma,\tau \in \RR, \tau\gtrless 0, \sigma\gtrless
  0 \}
  \cap \WF u=\emptyset
  $$
then
$$
\{ (r=t-t_0, \theta, t, \sigma, \tau,\eta=0 )\colon t>t_0,
\theta \in S^2, \sigma,\tau \in \RR, \tau \gtrless 0, \sigma \lessgtr 0\}
\cap \WF u=\emptyset.
$$
  \end{theorem}
Thus, no wavefront set \emph{arriving} at $r=0$ at time $t=t_0$ implies
no wavefront set \emph{emanating} from $r=0$ at time $t=t_0,$ and we
have established propagation on diffractive geodesics.  Moreover the
sign of $\tau$ is conserved in this interaction.

We will prove Theorem~\ref{theorem:diffractive1} by obtaining a stronger result, uniformly true across $r=0$, concerning the
\emph{propagation of $\bo$-wavefront set}.

\subsection{Propagation of b-regularity}

The following
treatment of the propagation of b-regularity is heavily influenced by
the work of Vasy in the context of manifolds with corners
\cite{Va:08}, which gave in turn a new perspective on previous results
of Melrose--Sj\"ostrand in the boundary case \cite{Melrose-Sjostrand1}, \cite{Melrose-Sjostrand2}.

The main propagation results take place inside the \emph{compressed
  characteristic set,} which is the appropriate extension of the
ordinary characteristic set to the boundary setting.  In coordinates associated to the canonical one-form
\begin{equation*}
  \tauu dt + \sigmau dr + \etau \cdot d\theta
\end{equation*}
on $T^{*}M$, $\Sigma$ is given by
\begin{equation*}
  \Sigma = \{ (r, \theta, t, \sigmau, \etau, \tauu) \mid \tauu^{2} - \sigmau^{2} - \frac{1}{r^{2}}\abs{\etau}^{2}\}.
\end{equation*}
The \emph{compressed characteristic set} $\Sigmadot$, originally due
to Melrose--Sj{\"o}strand~\cite{Melrose-Sjostrand1, Melrose-Sjostrand2}, is the image of the characteristic
set under the natural map $T^{*}M \to \Tbstar M$.  In the coordinates
associated to the canonical one-form
\begin{equation*}
  \tau dt + \sigma \frac{dr}{r} + \eta \cdot d\theta
\end{equation*}
on $\Tbstar M$, $\Sigmadot$ has the following form over $r=0$:
\begin{equation*}
  \Sigmadot |_{r=0} = \{ (r=0, \theta, t, \sigma=0, \eta = 0, \tau )
  \mid \theta \in S^{2}, \tau \neq 0\}.
\end{equation*}

We will obtain Theorem~\ref{theorem:diffractive1} by proving the following more precise
statement.  Recall from equation~\eqref{eth} that $\dop$ is a Dirac--Coulomb operator with
additionally a smooth vector potential, multiplied through by $\gamma^0$.
\begin{theorem}
  \label{thm:bpropagation}
  Assume $u$ is an admissible solution of $\dop u =0$, and assume that $\abs{\charge} <
  1/2$.

  For each $m$, $\WFb^{1,m}u \subset \Sigmadot$.  Away from $r=0$,
  $\WFb^{1,m}u$ is invariant under the bicharacteristic flow.

  Fix $\rho_{0} = \{ (r=0, \theta \in S^{2}, t_{0}, \sigma = 0,
  \eta_{0} = 0, \tau_{0})\} \subset \Sigmadot$ and let $U$ denote a
  neighborhood of $\rho_{0}$ in $\Sigmadot$.  If
  \begin{equation*}
    U \cap \{ \sigma/\tau  > 0\} \cap \WFb^{1,m}u = \emptyset,
  \end{equation*}
  then
  \begin{equation*}
    \rho_{0} \cap \WFb^{1,m}u = \emptyset.
  \end{equation*}
\end{theorem}
Note that the openness of the complement of $\WFb^{1,m} u$
 means that the theorem yields regularity at the outgoing points (where $\sigma/\tau<0$)
  sufficiently near $\rho_0.$

In fact, we prove a stronger statement for the inhomogeneous problem,
in which
\begin{equation*}
  \WFb^{1,m}u \subset \Sigmadot \cup \WFb^{m}(\dop u),
\end{equation*}
and if
\begin{equation*}
  U \cap \WFb^{0, m+1}(\dop u) = \emptyset
\end{equation*}
and
\begin{equation*}
  U \cap \{  \sigma/\tau > 0 \} \cap \WFb^{1,m}u = \emptyset,
\end{equation*}
then $\rho_{0} \cap \WFb^{1,m}u = \emptyset$, with analogous
statements with the additional factors included.

We also prove a statement about the propagation of coisotropic regularity.
\begin{theorem}
  \label{thm:b-module-reg}
  The same statements hold with $u$ replaced by $K^{\ell}u$ or
  $R^{\ell}u$, where $K$ is Dirac's $K$-operator and
  $R = (t-t_{0})D_{t} + rD_{r}$ is the scaling vector field.  More
  precisely, for each $\ell$, $\WFb^{1,m}(K^{j}u)$ and
  $\WFb^{1,m}(R^{j}u)$ are invariant under bicharacteristic flow away
  from $r=0$ for $j = 0 , \dots, \ell$ and if
  \begin{equation*}
    U \cap \{  \sigma/\tau> 0 \} \cap \WFb^{1,m} (S^{j}u) = \emptyset
  \end{equation*}
  for $S = K$ or $S = R$ and all $j \leq \ell$, and $\rho_{0} \cap
  \WFb^{m+1}(S^{j}\dop u)$ for $j = 0 , \dots, \ell$, then
  \begin{equation*}
    \rho_{0} \cap \WFb^{1,m}(S^{j}u) = \emptyset
  \end{equation*}
  for all $j \leq \ell$.
\end{theorem}

\begin{remark}
  The statement for $K$ provides a proof of the propagation of
  Lagrangian regularity through the singularity.  It immediately
  follows that a similar statement (with hypotheses modified as
  needed) holds for $K^{\ell}R^{k}u$; this shows that coisotropic
  regularity (in the $\bo$-sense) also propagates through the singularity.
\end{remark}

Since $\WFb^{1,m}$ is closed, this theorem \emph{implies}
Theorem~\ref{theorem:diffractive1} as follows:

\begin{proof}[Proof of Theorem~\ref{theorem:diffractive1} using Theorem~\ref{thm:bpropagation}]
Assuming the hypotheses of
Theorem~\ref{theorem:diffractive1}, we first can use ordinary
propagation of singularities and elliptic regularity over $M^\circ$ to conclude that a
neighborhood of
$$\{ (r=0, \theta, t=t_0, \sgn\sigma = \sgn\tau) \colon \theta \in S^2 \}
$$
over $M^\circ$ is disjoint from the wavefront set, since the backward
bicharacteristic flowout of any of these points lies in the region
where our hypotheses yield regularity, provided we take a sufficiently small
such neighborhood.  Without loss of generality, we will focus on the
component $\tau<0,$ with the other component to be treated mutatis mutandis.

Now we find that since ordinary and b
wavefront sets coincide for $r>0,$ over a neighborhood
of $(r=0, \theta \in S^2, t_0),$ $\WFb^{1,m} u\cap
\{r>0,\ \sgn \sigma\tau=1,\ \eta=0\}=\emptyset;$ since $\dot\Sigma \cap \{r=0\}\subset
\{\sigma=0\},$ this suffices to establish the existence of $U$ as in
the hypotheses of Theorem~\ref{thm:bpropagation}, where we have taken
fixed a sign of $\tau.$  Thus
Theorem~\ref{thm:bpropagation} implies that $\rho_0 \cap
\WFb^{1,m}=\emptyset;$ since $\WFb^{1,m}$ is closed, this implies the
existence of an open neighborhood of $\rho_0$ in $\dot \Sigma$ that is
disjoint from $\WFb^{1,m} u,$ and in particular, there is a
such neighborhood in $\dot \Sigma \cap \{\sigma>0, r>0\}.$

This is then the projection of an open neighborhood in $\Sigma,$ the
usual characteristic set, that is disjoint from $\WFb^{1,m} u$ and
where $\tau\sigma<0,$ and in particular contains a point in every
bicharacteristic $(r=t-t_0, t>t_0, \sgn\tau\sigma=-1, \eta=0);$
this completes the proof of
Theorem~\ref{theorem:diffractive1} (since $\WF u = \overline{\bigcup_m
  \WF^m u}$).
\end{proof}

We now proceed with the proof of Theorem~\ref{thm:bpropagation}.
To this end, we begin with preliminary estimates on commutators, with
a crucial role played by commutators between $\Box$ and b-operators
that are rotationally symmetric in the space variables.

\subsubsection{b-Commutators}

We record for our use below the form of the commutator of an invariant
(defined above in Definition~\ref{definition:invariant})
$\bo$-pseudodifferential operator with the second order operator $P$
and the first order operator $\dop$.

\begin{lemma}
  \label{lemma:commutator1}
  Let $C\in \Psib^{m}(M)$ be invariant, with principal symbol $c$
  scalar and real-valued.  Then, for $P$ satisfying the Klein-Gordon
  hypotheses of Section~\ref{section:KG},
  \begin{equation*}
    [P, C] = B_{0} \frac{1}{r^{2}}\Lap_{\theta} + B_{1},
  \end{equation*}
  where
  \begin{itemize}
  \item $B_{0} \in \Psib^{m-1}$ and 
  \item $B_{1}\in \Diff^{2}\!\Psib^{m-1} + \Diff^{1}\!\Psib^{m} +  \Psib^{m+1}$.
  \end{itemize}
  Both $B_{0}$ and $B_{1}$ are microsupported in $\WFb'C$.
\end{lemma}

\begin{proof}
  The term containing $B_{0}$ arises by commuting $C$ through the
  $\frac{1}{r^{2}}\Lap_{\theta}$ term in $P$.  The remaining terms in $P$
  contribute to the $B_{1}$ term; as
  \begin{equation*}
    P + \frac{1}{r^{2}}\Lap_{\theta} \in \DiffPsi^{0}, 
  \end{equation*}
  Lemma~\ref{lemma:differentialcommutator} shows that this commutator
  lies in $\DiffPsi^{m+1}$.
\end{proof}

\begin{lemma}
  \label{lemma:first-order-commutator}
  Let $C \in \Psib^{m}(M)$ be invariant, with principal symbol $c$
  scalar and real-valued.  Then
  \begin{equation}
    \label{Cfirstordercommutator}
    \frac{1}{i}[\dop, C] = A_{0} \left( \alpha_{r} \left( i \pd[r] +
        \frac{i}{r} - \frac{i}{r}\beta K\right) - \frac{\charge}{r}\right)
    +B_{0} + \alpha_{r}B_{1} + \bB_{2}\frac{1}{r} + \bB_{3}D_{r} +
    \bB_{4} + \bB_{5} \frac{1}{r} + \bB_{6} D_{r},
  \end{equation}
  where
  \begin{itemize}
  \item $A_{0} \in \Psib^{m-1}(M)$, with $\sigmab(A_{\bullet}) =
    -\pd[\sigma](c)$, 
  \item $B_{0} \in \Psib^{m}(M)$, with $\sigmab(B_{0}) = \pd[t](c)$,
  \item $B_{1} \in \Psib^{m}$, with $\sigmab(B_{1}) = \pd[r](c)$,
  \item $\bB_{2} \in \Psib^{m}(M)$, with $\supp \sigmab(\bB_{2})
    \subseteq \supp \pd[\eta](c)$,
  \item $\bB_{3}\in \Psib^{m-1}(M)$, with $\supp
    \sigmab(\bB_{3})\subseteq \supp\pd[\eta](c)$, 
  \item $\bB_{4} \in \Psib^{m-1}(M)$, and
  \item $\bB_{5}$, $\bB_{6} \in \Psib^{m-2}(M)$.
  \end{itemize}
\end{lemma}

\begin{remark}
  Non-scalar pseudodifferential operators are in bold in the expressions above; roman terms
  are scalar.
\end{remark}

\begin{proof}
  We write
  \begin{equation*}
    \dop = i \pd[t] - \frac{\charge}{r} + i \alpha_{r} \left( \pd[r] +
      \frac{1}{r} - \frac{1}{r}\beta K\right) - \alpha_0 V -\alpha_{j}A_{j},
  \end{equation*}
  where we use the convention that $\alpha_{0} = \id$.

  We begin with the angular term.  Because $\alpha_{r}$ and $K$ depend
  only on the angular variables, their commutators with the invariant
  operator $C$ are microsupported in the support of $\pd[\eta]c$.
  Writing
  \begin{equation*}
    \frac{1}{i} [ - \alpha_{r} \frac{i}{r} \beta K, C] = - \alpha_{r}
    \frac{1}{r}[\beta K, C] - \frac{1}{r}[\alpha_{r}, C] \beta K -
    [\frac{1}{r}, C]\alpha_{r}\beta K,
  \end{equation*}
  we see that the first two terms give contributions to $\bB_{2}$,
  while the last term yields the angular part of the $A_{0}$ term
  above.  (Indeed, we take this to define the operator $A_{0}$.)

  We now turn to the terms involving the commutator with $i \pd[t] -
  \frac{\charge}{r}$.    The $[\pd[t], C]$ term gives $B_{0}$,  while
  the $\frac{\charge}{r}$ term contributes to the $A_{0}$ and
  $\bB_{5}$ terms.  

  We now consider the term involving
  $i \alpha_{r} (\pd[r] + \frac{1}{r})$.  We observe that because
  $\alpha_{r}$ depends only on the angular variables, its commutator
  with $C$ is microsupported in the support of $\pd[\eta](c)$,
  yielding a contribution to the $\bB_{3}$ term.  Since
  \begin{equation*}
    - \frac{1}{i}[D_{r}, C] = S + TD_{r},
  \end{equation*}
  where
  \begin{equation*}
    S \in \Psib^{m}, \quad \sigmab(S) = \pd[r](c),
  \end{equation*}
  and
  \begin{equation*}
    T \in \Psib^{m-1}, \quad \sigmab(T) = \pd[\sigma](c), 
  \end{equation*}
  we see that the rest of this term yields contributions to the terms
  involving $A_{0}$, $B_{1}$, and $\bB_{6}$.

  Finally, the commutator of $-\alpha_0 V - \alpha^{j}\pot_{j}$ with $C$ yields
  the $\bB_{4}$ term.
\end{proof}

\subsubsection{Elliptic estimate}
\label{sec:elliptic}
The estimates in this section are very close to those in~\cite[Section
4]{Va:08}, hence we will be somewhat brief in the proofs; the main
difference here is in the potential terms, which need to be controlled
using the Hardy inequality.  Unlike in the proof of the hyperbolic
estimate in the next section, we work here with the second order
equation in order to obtain more direct control over the $H^{1}$
norm.  

\begin{lemma}
  \label{lemma:ellipticreg}
  If $\smallabs{\charge} < 1/2$ then for all $u \in H^{1}$,
  $\displaystyle\WFb^{1,m}u \subset \WFb^{-1, m}(Pu) \cup \dot\Sigma$.
\end{lemma}

Following the treatment in~\cite[Section 4]{Va:08}, we begin with a
lemma concerning the quadratic form associated to $P$.  (Cf.\
Lemma~4.2 of \cite{Va:08}.)

In what follows, we split $P$ as
\begin{equation*}
  P = P_{0} + \bR
\end{equation*}
with
\begin{equation*}
  P_{0} = - (\pd[0] + i\charge/r)^{2} + \sum \pd[j]^{2} - m^{2}- i \frac{\charge}{r^{2}}
  \begin{pmatrix}
    0 & \sigma_{r} \\ \sigma_{r} & 0 
  \end{pmatrix}.
\end{equation*}

\begin{lemma}
  \label{lemma:quadratic}
  Let $K\subset \Sbstar M$ be compact, $U\subset \Sbstar M$ open, $K
  \subset U$.  Let $A_{\lambda}$ be a bounded family of invariant
  elements in $\Psib^{s}$ with $\WFb'A_{\lambda} \subset K$ (in the
  sense of uniform wavefront set of families), and $A_{\lambda} \in
  \Psib^{s-1}$ for $\lambda \in (0,1)$.  Then there exist $G \in
  \Psib^{s-1/2}$, $\tG \in \Psib^{s}$, both microsupported in $U$,
  and $C_{0}$ so that for all $\epsilon > 0$, $\lambda \in (0,1)$, $u
  \in H^{1}$ with $\WFb^{1,s-1/2}u \cap U = \emptyset$, $\WFb^{-1,
    s}(Pu) \cap U = \emptyset$,
  \begin{align*}
    &\abs{\norm{(D_{t}+\charge/r)A_{\lambda}u}^{2} - \norm{\grad
    A_{\lambda}u}^{2} - m^{2}\norm{A_{\lambda}u}^{2} + \Re \langle \bR
      A_{\lambda}u, A_{\lambda}u \rangle } \\
    &\leq C_{0} \left( \epsilon \norm{A_{\lambda}u}_{H^{1}}^{2} +
      \norm{u}_{H^{1}}^{2} + \norm{Gu}_{H^{1}}^{2} +
      \epsilon^{-1}\norm{Pu}_{H^{-1}}^{2} + \epsilon^{-1}\norm{\tG Pu}_{H^{-1}}^{2}\right).
  \end{align*}
  The estimate is \emph{uniform} for bounded $\charge$ (which is not
  required to be small).  
\end{lemma}

\begin{remark}\mbox{}
  \begin{itemize}
  \item The LHS of the inequality is given by the absolute value of
    the $\Re \langle PA_{\lambda}u, A_{\lambda}u\rangle$; the
    non-scalar term in $P$ is anti-self-adjoint, hence does not
    contribute.
  \item If $A_{\lambda}$ commuted with $P$ the $G$ term would not
    appear; as it is, this term is lower order than $A_{\lambda}$
    since it arises as a commutator.
  \end{itemize}
\end{remark}

\begin{proof}
  Fix $G, \tG$ of the appropriate order, microsupported in $U$, so
  that $\sigmab(G), \sigmab(\tG) \equiv 1$ on $K$.

  The pairing
  \begin{equation*}
    \Re \langle PA_{\lambda}u, A_{\lambda}u\rangle
  \end{equation*}
  is finite for all $\lambda > 0$ by our wavefront set assumption,
  which implies that $PA_{\lambda}u \in H^{-1}$ and $A_{\lambda}u \in
  H^{1}$.  First write
  \begin{equation*}
        \abs{\Re \ang{PA_{\lambda}u, A_{\lambda}u}} \leq
                                                  \abs{\ang{[P,A_{\lambda}]u,
                                                  A_{\lambda}u}} +
                                                  \abs{\ang{A_{\lambda}Pu,
                                                  A_{\lambda}u}} .
  \end{equation*}

  We first estimate the term
  \begin{equation*}
    \abs{\ang{A_{\lambda} Pu, A_{\lambda}u}}.
  \end{equation*}
  Indeed, we observe that
  \begin{equation*}
    \abs{\ang{A_{\lambda} Pu, A_{\lambda}u}} \leq
    \norm{A_{\lambda}Pu}_{H^{-1}}\norm{A_{\lambda}u}_{H^{1}} \leq
    \epsilon\norm{A_{\lambda}u}_{H^{1}}^{2} + \epsilon^{-1}\norm{A_{\lambda}Pu}_{H^{-1}}^{2}.
  \end{equation*}
  Elliptic regularity for $\tG$ then shows that
  \begin{equation*}
    \abs{\ang{A_{\lambda} Pu, A_{\lambda}u}} \leq \epsilon
    \norm{A_{\lambda}u}_{H^{1}}^{2} + C\epsilon^{-1}\left(
      \norm{Pu}_{H^{-1}}^{2} + \norm{\tG Pu}_{H^{-1}}^{2}\right).
  \end{equation*}

  We now turn our attention to the commutator term.  Indeed,
  Lemma~\ref{lemma:commutator1} allows us to write
  \begin{equation*}
    \ang{[P, A_{\lambda}] u, A_{\lambda}u} =
    \ang{\frac{1}{r^{2}}\Lap_{\theta}B_{0}u, A_{\lambda}u} +
    \ang{B_{1}u, A_{\lambda}u},
  \end{equation*}
  where $B_{0}\in \Psib^{s-1}$ and $B_{1} \in
    \DiffPsi^{s-1}+\operatorname{Diff}^1\!\Psi_b^s+\Psi_b^{s+1},$ both
    satisfying uniform (in $\lambda$) estimates in these spaces.
  
  Lemmas~\ref{lemma:order-shifting}
  and~\ref{lemma:elliptic-estimation} show that we may bound
  \begin{equation*}
    \abs{\ang{[P,A_{\lambda}]u, A_{\lambda}u}} \lesssim
    \norm{u}_{H^{1}}^{2} + \norm{Gu}_{H^{1}}^{2},
  \end{equation*}
  finishing the proof.
\end{proof}

\begin{proof}[Proof of Lemma~\ref{lemma:ellipticreg}]
  (Cf.\ the proof of~\cite[Proposition 4.6]{Va:08}.)  We aim to show
  that if $\WFb' A \cap \Sigmadot = \emptyset$ and
  $\WFb^{-1,m}(Pu)\cap \WFb'A =\emptyset$, then $Au \in H^{1}$.  We in
  fact show this iteratively, assuming by induction that
  $\WFb^{1,s-1/2}u$ is disjoint from a(n arbitrarily small
  neighborhood of) $\WFb'A$ and then showing $Au \in H^{1}$.  To pass
  to $s=\infty$, one must guarantee that the supports of the operators
  in each iteration do not shrink too quickly, but this can be
  guaranteed as in the end of the proof of~\cite[Proposition 6.2]{Va:08}.

  We will use the notation
  $$
\hat\sigma=\frac{\sigma}{\smallabs{\tau}},\quad \hat\eta=\frac{\eta}{\smallabs{\tau}}
    $$
    in discussing symbol constructions below.

  Since $\WFb'A \cap \Sigmadot = \emptyset$, without loss of
  generality (since the lemma is standard over $M^{\circ}$),
  $\hat{\sigma}^{2}+ \abs{\hat{\eta}}^{2}\geq \epsilon^{2} > 0$ on $\WFb'
  A$; moreover, by a partition of unity in $x$ (again using elliptic
  regularity over $M^{\circ}$), we may take $r < \delta$ over
  $\WFb'A$, where we may specify $\delta$ independently from
  $\epsilon$ above.  Now we let\footnote{We assume our quantization is
    arranged so that it yields properly supported operators.}
  \begin{equation*}
    A_{\lambda} \equiv \Op_b \left( (1 + \lambda (\tau^{2} + \sigma^{2}+
      \abs{\eta}^{2}))^{-1}\right) A,
  \end{equation*}
  so that $A_{\lambda}$ is uniformly bounded in $\Psib^{s}$ and
  converges to $A$ in the topology of $\Psib^{s+0}$, while for each
  $\lambda > 0$, $A_{\lambda} \in \Psib^{s-2}$.  We may apply
  Lemma~\ref{lemma:quadratic} to such an $A$, so that for all
  $\epsilon ' > 0$,
  \begin{align}
    \label{qform2}
   & - \norm{(D_{t}+\charge/r)A_{\lambda}u}^{2} + \norm{\grad A_{\lambda}u}^{2}
    + m^{2}\norm{A_{\lambda}u}^{2} - \abs{\ang{\bR A_{\lambda}u,
     A_{\lambda}u}} \\
    &\leq C_{0} \left( \epsilon'\norm{A_{\lambda}u}_{H^{1}}^{2} +
      \norm{u}_{H_{1}}^{2} + \norm{Gu}_{H_{1}}^{2} +
      (\epsilon')^{-1}\norm{Pu}_{H^{-1}}^{2} + (\epsilon')^{-1}
      \norm{\tG Pu}_{H^{-1}}^{2}\right).
  \end{align}
  Since $\tau^{2} < \epsilon^{-2} (\sigma^{2} + \abs{\eta}^{2})$ and
  $r < \delta$ on $\WFb'A$, we estimate
  \begin{align*}
    \norm{D_{t}A_{\lambda}u}^{2} &\leq
                                   \ang{\epsilon^{-2}\Op(\sigma^{2} +
                                   \abs{\eta}^{2}) A_{\lambda}u ,
                                   A_{\lambda}u} +
                                   \norm{Gu}_{H^{1}}^{2} \\
    & = \epsilon^{-2}\left( \norm{(rD_{r})A_{\lambda}u}^{2} +
      \norm{\grad_{\theta}A_{\lambda}u}^{2}\right) +
      \norm{Gu}_{H^{1}}^{2} \\
    &\leq \delta^{2} \epsilon^{-2}\norm{\grad A_{\lambda}u}^{2} +
      \norm{Gu}_{H^{1}}^{2} \\
    &\leq C \delta^{2} \epsilon^{-2}\norm{A_{\lambda}u}_{H^{1}}^{2} + \norm{Gu}_{H^{1}}^{2}.
  \end{align*}
  Here again $G \in \Psib^{s-1/2}$ is an error term (which we allow to change from line to line as needed); we
  use it to estimate terms of the form $\norm{Bu}_{L^{2}}^{2}$ with $B
  \in \Psib^{s+1/2}$.

  We also recall from~\eqref{Hardyconstant} that
  \begin{equation*}
    \norm{r^{-1}A_{\lambda}u}^{2} \leq 4
    \norm{A_{\lambda}u}_{H^{1}}^{2};
  \end{equation*}
  thus for any $\ep'>0,$
  \begin{equation*}
    \norm{(D_{t} + V)A_{\lambda}u}^{2}\leq C \delta^{2} \epsilon^{-2}
    \norm{A_{\lambda}u}_{H^{1}}^{2} + (4 + \epsilon ') \charge^{2} \norm{\grad
      A_{\lambda}u}^{2} + \norm{Gu}_{H^{1}}^{2}.
  \end{equation*}
  We also use repeatedly the fact that $\norm{A_{\lambda}u}\leq
  C\norm{Gu}_{H^{1}}$ together with Cauchy--Schwarz to estimate
  \begin{equation*}
    \abs{\ang{\bR A_{\lambda}u, A_{\lambda}u}} \leq \epsilon'
    \norm{A_{\lambda}u}_{H^{1}}^{2} + C\norm{Gu}_{H^{1}}^{2}.
  \end{equation*}
  (The constant on the right side depends on both $\charge$ and
  $\epsilon'$.)

  Adding $\norm{D_{t}A_{\lambda}u}^{2} + \norm{(D_{t}+V)A_{\lambda}u}^{2}$ to equation~\eqref{qform2} now yields
  \begin{align}
    \label{qform3}
    \norm{D_{t}A_{\lambda}u}^{2} + \norm{\grad A_{\lambda}u}^{2} &\leq
                                                                   (C\delta^{2}\epsilon^{-2}
                                                                   +
                                                                   2\epsilon')\norm{A_{\lambda}u}^{2}
                                                                   +
                                                                   (4 +\epsilon')
                                                                   \charge^{2}
                                                                   \norm{\grad
                                                                   A_{\lambda}u}^{2}
    \\
    &\quad +
                                                                   C_{0}\left(\norm{u}_{H^{1}}^{2}
                                                                   +
                                                                   \norm{Gu}_{H^{1}}^{2}
                                                                   +
                                                                   (\epsilon')^{-1}
                                                                   \norm{Pu}_{H^{-1}}^{2}
                                                                   +
                                                                   (\epsilon')^{-1}\norm{\tG
      Pu}_{H^{-1}}^{2}\right). \notag
  \end{align}
  Assuming now that $\abs{\charge}< 1/2$, taking $\epsilon'$ and
  $\delta$ sufficiently small (and dropping $\epsilon'$-dependence of
  the constants on the right side), we absorb the
  $\norm{\grad A_{\lambda}u}^{2}$ and
  $\norm{A_{\lambda}u}_{H^{1}}^{2}$ terms on the right into the left
  side.  (For the latter term, we recall that up to
  $\norm{A_{\lambda} u}_{L^{2}}^{2}$, which is controlled by
  $\norm{Gu}_{H^{1}}^{2}$, $\norm{\grad A_{\lambda}u}^{2}$ is
  comparable to the squared $H^{1}$ norm of $A_{\lambda}u.$)

  We thus obtain
  \begin{equation}\label{finalelliptic}
    \norm{A_{\lambda}u}_{H^{1}}^{2} \leq C \left( \norm{u}_{H^{1}}^{2}
      + \norm{Gu}_{H^{1}}^{2} + \norm{Pu}_{H^{-1}}^{2}+ \norm{\tG Pu}_{H^{-1}}^{2}\right).
  \end{equation}
  The right side is uniformly bounded as $\lambda \downarrow 0$ by our
  inductive assumption.  Now taking $\lambda \to 0$ and employing a
  standard weak-convergence argument (see, e.g.,~\cite[Lemma 3.7]{Va:08}) shows that
  $Au \in H^{1}$.  This concludes the proof of Lemma~\ref{lemma:ellipticreg}.
\end{proof}

We now record two corollaries of the previous lemma; the first is
elliptic regularity for $\dop$:
\begin{corollary}
  \label{lemma:first-order}
  If $\abs{\charge} < 1/2$, then for all $u \in H^{1}$, $\WFb^{1,m}u
  \subset \WFb^{m}(\dop u)\cup \Sigmadot$.

  More precisely, if $A\in \Psib^{m}$ is properly supported and
  microsupported near $\rho_{0} \notin \Sigmadot$, then there are
  $G\in \Psib^{m-1}$ and $\tG \in \Psib^{m}$ also microsupported in
  $\Sigmadot^c$ so that
  \begin{equation*}
    \norm{Au}_{H^{1}} \leq C \left( \norm{u}_{H^{1}} +
      \norm{Gu}_{H^{1}} +  \norm{\tG \dop u}_{L^{2}}\right).
  \end{equation*}
\end{corollary}

\begin{proof}
  The final estimate \eqref{finalelliptic} in the proof of
  Lemma~\ref{lemma:ellipticreg} shows that we may bound
  \begin{equation*}
    \norm{Au}_{H^{1}} \leq C \left( \norm{u}_{H^{1}} +
      \norm{G_{1}u}_{H^{1}} + \norm{Pu}_{H^{-1}} + \norm{\tG Pu}_{H^{-1}}\right),
  \end{equation*}
  for $G_{1} \in \Psib^{m-1/2}$.  We first estimate $\norm{Pu}_{H^{-1}}$ and $\norm{\tG
    Pu}_{H^{-1}}$ in terms of $\norm{\dop u}_{L^{2}}$.

  We recall that we may write $P = \tL \dop$, where
  \begin{equation*}
    \tL = (i \dirac_\pot +m)\gamma^{0},
  \end{equation*}
  which maps $L^{2}\to H^{-1}$ continuously.  We may therefore bound
  \begin{equation*}
    \norm{Pu}_{H^{-1}} \leq C \norm{\dop u}_{L^{2}}.
  \end{equation*}
  Turning to $\norm{\tG Pu}_{H^{-1}}$, we write
  \begin{equation*}
    \tG Pu = \tG \tL \dop u =\tL \tG \dop u + [\tG, \tL] \dop u.
  \end{equation*}
  As $\norm{\tL \tG \dop u}_{H^{-1}}\leq C \norm{\tG \dop u}_{L^{2}}$, we turn
  our attention to $[\tG , \tL]$.  As $\tL \in r^{-1}\Psib^{1}$,
  Lemma~\ref{lemma:differentialcommutator} and basic properties of the
  $\bo$-calculus show that $[\tG, \tL] \in r^{-1}\Psib^{m}$.  Elliptic
  regularity of (a slightly enlarged) $\tG '$ then shows that
  \begin{equation*}
    \norm{[\tG, \tL ]\dop u}_{H^{-1}} \leq C \left( \norm{\dop u}_{L^{2}} +
      \norm{\tG '\dop u}_{L^{2}}\right).
  \end{equation*}
    
  We now repeat the whole argument up to this point with $G_{1}$
  replacing $A$; this allows us to replace (at the cost of slightly
  enlarging the microsupports) the operator $G_{1} \in \Psib^{m-1/2}$
  with $G \in \Psib^{m-1}$.
\end{proof}

\begin{remark}
  By iteration, we may replace $G \in \Psib^{m-1}$ in the statement of
  the above corollary by an operator of any order, though we do not
  need this stronger statement below.
\end{remark}

The second corollary has the same proof as
Lemma~\ref{lemma:ellipticreg} without an estimate on
$\norm{D_{t}A_{\lambda}u}$:
\begin{corollary}
  \label{lemma:control-of-H1-by-dt}
  If $\abs{\charge}<1/2$ and $A\in \Psib^{m}$ is invariant and
  properly supported, then for $G \in \Psib^{m-1}$ and $\tG \in
  \Psib^{m}$ with $\WFb'(A) \subset \liptic G \cap \liptic \tG$, we have
  \begin{equation*}
    \norm{Au}_{H^{1}} \leq C \left( \norm{D_{t}Au} + \norm{Gu}_{H^{1}} +
      \norm{\tG \dop u}_{H^{1}} + \norm{u}_{H^{1}}\right).
  \end{equation*} 
\end{corollary}

\subsubsection{Proof of Theorems~\ref{thm:bpropagation} and~\ref{thm:b-module-reg}}
\label{sec:hyperbolic-diffractive}

We now turn our attention to the proof of the $\bo$-propagation
theorems.
We first record a consequence of the elliptic estimates of the
previous section:
\begin{lemma}
  \label{lemma:H1toL2}
  Suppose $u\in H^{1}$, $\dop u =0$.  Then
  \begin{equation*}
    \left( \WFb^{1,m}u\right)^{c} = \left\{ \rho \in \Tbstar M \colon
      \text{ there exists } A \in \Psib^{m+1} , \text{ elliptic at
      }\rho, Au \in L^{2}\right\}.
  \end{equation*}
  (Cf.\ Lemma 6.1 of \cite{Va:08}.)

  More precisely, if $u \in H^{1}$ and $\rho_{0} \notin
  \WFb^{m+1}(\dop u)$,
  then
  \begin{equation*}
    \rho_{0} \in \WFb^{1,m}u \text{ if and only if }\rho_{0} \in \WFb^{m+1}u.
  \end{equation*}
\end{lemma}

\begin{proof}
Suppose $\rho_0 \notin \WFb^{1,m} u.$
  We may use a microlocal partition of unity in the $b$-calculus to break
  $u$ into pieces on each of  which one of the the operators $r D_r$
  $D_t,$ or $D_{\theta_j}$ is b-elliptic.  If $A \in \Psib^{m+1}$ and
  $G \in \Psib^m$ is elliptic on $\WFb' u,$ we thus obtain by
  microlocal ellipticity
$$
\norm{Au }^2\lesssim \norm{rDr G u}^2 + \norm{D_t G u}^2 +
\norm{\nabla_\theta G u}^2 + \norm{ u}^2_{H^1}\lesssim \norm{G u}^2_{H^1}+ \norm{ u}^2_{H^1}.
$$
and we obtain one direction of the lemma.

  The other direction of the lemma follows immediately from Corollary~\ref{lemma:control-of-H1-by-dt}.
\end{proof}

We now turn to the proof of Theorem~\ref{thm:bpropagation}.  Let
us first consider the case when $M=0$ and let $U$ denote a neighborhood
of $\rho_{0}$ in $\Sigmadot$ with
\begin{equation*}
  U \cap \{ \sigma > 0\} \cap \WFb^{s+1/2}u
  =  U \cap \WFb^{s+1/2}(\dop u) =\emptyset.
\end{equation*}
For our inductive hypothesis we assume that $\rho_{0} \notin
\WFb^{s}u$; we aim to show $\rho_{0} \notin \WFb^{s+1/2}u$.

Let $\omega = r^{2} + (t-t_{0})^{2}$, and let
\begin{equation*}
  \phi = - \hat{\sigma} + \frac{1}{\beta^{2}\delta} \omega.
\end{equation*}
Fix a small neighborhood $U$ of $(t = t_{0}, x= 0)$ in $\Sbstar M$ and
choose cutoff functions $\chi_{0}$, $\chi _{1}$, and $\chi _{2}$ with
the following properties:
\begin{itemize}
\item $\chi_{0}$ is supported in $[0, \infty)$, with $\chi_{0}(s) =
  \exp (-1/s)$ for $s > 0$,
\item $\chi_{1}$ is supported in $[0, \infty)$, with $\chi_{1}(s) = 1$
  for $s\geq 1$ and $\chi' \geq 0$, and 
\item $\chi_{2}$ is supported in $[-2c_{1}, 2c_{1}]$, and is equal to
  $1$ on $[-c_{1}, c_{1}]$.
\end{itemize}
Here $c_{1}$ is chosen so that $\hat{\sigma}^{2} + \hat{\eta}^{2} <
c_{1} < 2$ in $\Sigmadot \cap U$.

Now set
\begin{equation}
  \label{eq:hypcommutant}
  a = |\tau|^{s+1/2} \chi_{0}(2 - \phi / \delta)\chi_{1}(2 -
  \hat{\sigma}/\delta)\chi_{2}(\hat{\sigma}^{2}+\abs{\hat{\eta}}^{2})1_{\sgn\tau
  = \sgn \tau_{0}}
\end{equation}
and let $A$ be its quantization to an invariant element of
$\Psib^{s+1/2}$.  Note that
\begin{equation}
  \label{suppa}
  \supp a \subset \{ \abs{\hat{\sigma}} < 2\delta, \omega < 4 \delta^{2}\beta^{2}\},
\end{equation}
hence the support of $a$ in $\Tbstar M$ can be taken to be inside any
desired neighborhood of $\rho_{0}$.

In the following symbol construction and subsequent argument, we will
omit a standard regularization argument, described in detail in
\cite[(6.19) et seq.]{Va:08}.
\begin{lemma}
  \label{lemma:commutator-first-order}
  For $A$ defined as above,
  \begin{equation}
    \label{commutator-first-order}
    i^{-1} [ \dop , A^{*}A] = \tilde{R}\dop  - \sgn (\tau_{0})Q^{*}Q + \bR_{1}\frac{1}{r} +
    \bR_{2}D_{r} + \bR_{3}\frac{1}{r}\beta K + \bR + B_{0} + \alpha_{r}B_{1}
    + E' + E'',
  \end{equation}
  where
  \begin{itemize}
  \item $Q \in \Psib^{s+1/2}$ is invariant and self-adjoint with
    \begin{equation*}
      \sigmab(Q) =
      \sqrt{2}|\tau|^{s+1/2}\delta^{-1/2}(\chi_{0}'\chi_{0})^{1/2}\chi_{1}
      \chi_{2} 1_{\sgn\tau = \sgn\tau_{0}},
    \end{equation*}
  \item $\tilde{R} \in \Psib^{2s}$, 
  \item $\bR_j \in \Psib^{2s-1}$, 
  \item $\bR \in \Psib^{2s}$,
  \item $B_{0}, B_{1} \in \Psib^{2s+1}$ with
    $\abs{\sigmab(B_{\bullet})}$ equal to an order zero symbol times $C\beta^{-1}\sigmab(Q)^{2}$, 
  \item $E' \in \Psib^{2s+1}$ with $\WFb'E' \subset \{ \delta \leq
    \hat{\sigma}\leq 2\delta, \omega \leq 4\beta^{2}\delta\}$, and 
  \item $E'' \in \frac{1}{r}\Psib^{2s+1} + \Diff^{1}\!\Psib^{2s} + \Psib^{2s+1}$, with $\WFb'E'' \cap
    \Sigmadot = \emptyset$.
  \end{itemize}
  All terms above have microsupport within $\supp a$.  
\end{lemma}

\begin{proof}
  We apply Lemma~\ref{lemma:first-order-commutator} and employ the
  notation therein.  The term
  $A_0$ arising there has principal symbol $-\pd[\sigma](a^{2})$
  and arises from $\dop$ being nearly homogeneous in $r$ of degree
  $-1$.  We may rewrite the $A_0$ term in
  \eqref{Cfirstordercommutator} as $A_0(\dop+D_t),$ modulo $A_0$ times
  smooth lower-order terms (which are then absorbed into $\bR$).
We now split the symbol of $A_{0}$
  into three terms: those terms where the derivative falls on
  $\chi_{0}$ can be written in the form $\tQ^{2}D_{t}$, which we write
  as the product of $\sgn (\tau_{0})$ times squares $Q^{2}$ modulo a
  lower order error we that we absorb into $\bR$.  Meanwhile, those terms
  where the $\sigma$ derivative falls on $\chi_{1}$ we absorb into
  $E'$ and those on which it falls on $\chi_{2}$ we absorb into
  $E''$.  Thus, modulo further commutators (again absorbed into the
  error terms $\bR, \bR_j$) we have written the first term on the RHS of
  \eqref{Cfirstordercommutator} as $\tilde{R}\dop  - \sgn (\tau_{0})Q^{*}Q.$

  The $B_{1}$ term arising in Lemma~\ref{lemma:first-order-commutator}
  enjoys the asserted symbol bounds because $r$ derivatives on $a^{2}$
  may only fall on the $\chi_{0}$ term, giving
  \begin{equation*}
    2 |\tau|^{2s+1} (\chi_{0}'\chi_{0})\chi_{1}^{2}\chi_{2}^{2} (-2r)(\beta^{-2}\delta^{-2});
  \end{equation*}
  since $0 \leq r \leq 2\beta \delta$ on the support of $a$, this term
  is estimated by a multiple of
  $\beta^{-1}\delta^{-1}|\tau|^{2s+1}\chi_{0}'\chi_{0}\chi_{1}^{2}\chi_{2}^{2}$,
  which in turn is a multiple of $\beta^{-1}\sigmab(Q)^{2}$.
  Likewise, the $B_{0}$ term in
  Lemma~\ref{lemma:first-order-commutator} becomes the $B_{0}$ term
  here and is estimated similarly, as the $t$ derivative may also only
  hit the $\chi_{0}$ term.

  Finally, the $\bB_{2}$ and $\bB_{3}$ terms from
  Lemma~\ref{lemma:first-order-commutator} have symbols proportional
  to $\pd[\eta](a^{2})$, so the derivative must fall on $\chi_{2}$ and
  these terms are absorbed into $E''$.  The $\bB_{5}$ term is also
  absorbed into $\bR$.
\end{proof}

We now return to the main argument.  We pair $i^{-1}[\dop , A^{*}A]u$ with
$u$ and employ a regularization argument as in the elliptic setting.
On the one hand, we may bound
\begin{equation*}
  \abs{\ang{[\dop , A^{*}A]u, u}} = \abs{\ang{Au, A\dop u} -
    \ang{A\dop u, Au}}
  \leq 2 \norm{Au}\norm{A\dop u} \leq \epsilon \norm{Au}^{2} +
  \epsilon^{-1}\norm{A\dop u}^{2}.
\end{equation*}

On the other hand, we apply Lemma~\ref{lemma:commutator-first-order}.

The main term is
$-\sgn (\tau_{0})\ang{Q^{*}Qu, u} = -\sgn(\tau_{0})\norm{Qu}^{2}$,
which has a definite sign.  We may then bound
\begin{align*}
  \norm{Qu}^{2} &\leq \epsilon \norm{Au}^{2} +
  \epsilon^{-1}\norm{A\dop u}^{2} + \abs{\ang{\tR \dop u, u}} +
  \abs{\ang{\bR_{1}\frac{1}{r}u, u}} + \abs{\ang{\bR_{2}D_{r}u, u}} \\
                         &\quad + \abs{\ang{\bR_{3}\frac{1}{r}\beta K u, u}} + \abs{\ang{\bR u, u}} +
                           \abs{\ang{B_{0}u, u}} + \abs{\ang{\alpha_{r}B_{1}u, u}} \\
                         &\quad + \abs{\ang{E' u, u}} + \abs{\ang{E''u,u}}.
\end{align*}

As $\tR \in \Psib^{2s}$, the $\abs{\ang{\tR \dop u, u}}$ term is bounded
by $\norm{G_{s}\dop u}\norm{G_{s}u}$ for some $G_{s}\in \Psib^{s}$.  Similarly, the
terms involving $\bR_{j}$ can be estimated by
$\norm{G_{s-1}u}_{H^{1}}\norm{G_{s}u}$ for some $G_{s-1}\in
\Psib^{s-1}$ and $G_{s} \in \Psib^{s}$.  As $\bR \in \Psib^{2s}$, the
term $\ang{\bR u, u}$ is bounded by $\norm{G_{s}u}^{2}$.  This leaves
the terms involving $B_{0}$ and $B_{1}$ as well as the $E'$ and $E''$ terms.

The following lemma allows us to bound the terms involving $B_{0}$ and
$B_{1}$:
\begin{lemma}\label{lemma:Bestimate}
  There exists $G \in \Psib^{s}$ with $\WFb'G \cap \WFb^{s}u =
  \emptyset$ so that for $j = 0, 1$, 
  \begin{equation*}
    \abs{\ang{B_{j}u, u}} \leq C \beta^{-1} \norm{Qu}^{2}
    + C \norm{Gu}_{L^{2}}^{2} + C \norm{u}_{H^{1}}^{2}.
  \end{equation*}
\end{lemma}

\begin{proof}[Proof of Lemma~\ref{lemma:Bestimate}]
  By the pseudodifferential calculus, we may write $B_{j} = QC_{1}C_{2}Q + R$,
  where $C_{i} \in \Psib^{0}$ satisfies $\abs{\sigmab(C_{i})} \leq C
  \beta^{-1/2}$ and $R \in \Psib^{2s}$, and
  \begin{equation*}
    \left( \WFb'R \cup \WFb'C_{i}\right) \cap \WFb^{s}u = \emptyset.
  \end{equation*}
  For any $w \in L^{2}$ with $\WFb^{0}w \cap \WFb'C_{i} = \emptyset$,
  our symbol estimate gives 
\begin{equation}\label{someequation}
  \begin{aligned}
    \norm{C_{i}w}_{L^{2}} &\leq C\beta^{-1/2} \norm{G_{0}w}_{L^{2}}
    +C\norm{G_{0}w}_{H_{\bo, g}^{-1}} + C \norm{w}_{H_{\bo, g}^{-N}}\\
  \end{aligned}
  \end{equation}
  for some microlocalizer $G_{0}\in \Psib^{0}$ with
  $\WFb'(1-G_{0})\cap \WFb'C_{i} = \emptyset$.  In particular, then,
  setting $w = Qu$ yields
  \begin{equation*}
    \norm{C_{i}Qu}\leq C \beta^{-1/2}\norm{Qu}_{H^{1}} + C
    \norm{G u}_{L^{2}} + C \norm{u}_{H^{1}},
  \end{equation*}
  for $G = G_{0}Q$ as in the statement of the lemma.

  An application of Cauchy--Schwarz to $\ang{B_{j}u, u}$ then yields
  the stated estimate and concludes the proof of Lemma~\ref{lemma:Bestimate}.
\end{proof}

The term involving $E'$ is bounded by $\norm{G_{s+1/2}u}^{2}$, where
$G_{s+1/2}\in \Psib^{s+1/2}$ has $\WFb'G_{s+1/2} \subset \{ \delta
\leq \hat{\sigma} \leq 2\delta, \omega \leq 4 \beta^{2}\delta^{2}\}$.
The hypothesis that $U \cap \{ \sigma > 0\} \cap \WFb'^{1, s-1/2}u =
\emptyset$ implies that this term is finite.

Finally, we estimate the term involving $E''$.  As the microsupport of
$E''$ is contained in the elliptic set of $\dop$, we may use
Corollary~\ref{lemma:first-order} to bound this term by
\begin{equation*}
  C \left( \norm{u}_{H^{1}}^{2} + \norm{G_{s-1}u}_{H^{1}}^{2} +
     \norm{\tG_{s} \dop u}_{L^{2}}^{2}\right) ,
\end{equation*}
where $G_{s-1} \in \Psib^{s-1}$ and $\tG_{s} \in \Psib^{s}$ are
microsupported in the elliptic region within $U$.

Thus,
  \begin{equation}\label{finalpropagation}
  \norm{Qu}^{2} \leq \epsilon \norm{Au}^{2} +
  \epsilon^{-1}\norm{A\dop u}^{2}+ \norm{\tG_{s} \dop u}_{L^{2}}^{2}+\text{finite},
\end{equation}
where the terms labeled finite have been estimated by our inductive
assumptions on $u.$  Since $\sigma_b(A)/\sigma_b(Q) \leq C,$ we
may absorb the first term on the right into the left side modulo finite terms,
provided $\epsilon$ is sufficiently small; $\norm{Q u}$ is thus
bounded.  As $Q$ is elliptic at $\rho_{0}$, $\rho_{0}
\notin \WFb^{s+1/2}u$ (and hence, by Lemma~\ref{lemma:H1toL2}, not in
$\WFb^{1,s-1/2}u$).  This completes the proof of
Theorem~\ref{thm:bpropagation}.

We now turn to the proof of Theorem~\ref{thm:b-module-reg}.  The
  arguments of Section~\ref{sec:interior-coiso} imply the propagation
  result away from the $r=0$, so we need only prove the result through
  the singularity.  We first
describe the commutators of $\dop$ with $R$ and $K$:

\begin{lemma}
  \label{lemma: b-module-commutators}
The commutators of $\dop$ with $R^{\ell}$ 
  $K^{\ell}$ are as follows:
  \begin{enumerate}
  \item $[\dop, R^{\ell}]$ can be written as a linear combination of
    $\dop R^{j}$ (or, indeed, $R^{j}\dop$) and $R^{j}\bF_{j}$ (or
    $\bF_{j}R^{j}$), where $j =0, 1, \dots, \ell-1$ and $\bF_{j}\in
    C^{\infty}$ (but not necessarily scalar).
  \item $[\dop, K^{\ell}]$ is a linear combination of terms of the
    form $K^{j}\bB K^{\ell-1-j}$, where $j = 0, 1, \dots, \ell-1$ and
    $\bB \in \Diffb^{1}$ only differentiates in the angular variables.
  \end{enumerate}
\end{lemma}

\begin{proof}
  To prove the first statement, we write
  $\dop = i \gamma^{0}\dirac_{\charge/r} + \bR$, where $\dirac_{\charge/r}$ is the
  Dirac operator with potential $\pot = (\charge/r, 0, 0, 0)$ and
  $\bR = -\sum_{\mu=0}^{3} \alpha_{\mu}A_{\mu}$.  Because
  $\dirac_{\charge/r}$ is homogeneous of degree $-1$ in $(t,r)$, we
  see that
  \begin{equation*}
    [\dop, R] = \frac{1}{i} (\dop - \bR) + [\bR, R] = \frac{1}{i} \dop
    + \bF_{0}.
  \end{equation*}
  We then observe that
  \begin{equation*}
    [\dop, R^{k}] = [\dop, R]R^{k-1} + R[\dop, R^{k-1}].
  \end{equation*}
  The first term on the right is then of the correct form by our
  calculation of $[\dop, R]$, while the second term is a linear
  combination of terms of the form $R \dop R^{j}$ and $R \bF R^{j}$,
  where $j = 0, 1, \dots, k-2$ by the inductive hypothesis.  As we can
  commute $R$ with $\dop$ and $\bF_{j}$ at the cost of lower order
  terms of the same form, this proves the first statement.

  We prove the second statement similarly.  Because $K$ commutes with
  $\dirac_{\charge/r}$ and $\gamma^{0}$, we can see that
  \begin{equation*}
    [\dop, K^{\ell}] = [ \bR , K^{\ell}] =
    \sum_{j=0}^{\ell-1}K^{j}[\bR, K] K^{\ell-1-j}.
  \end{equation*}
  As $\bR$ is non-scalar, $[\bR, K]\in \Diffb^{1}$ is only a first order
  differential operator, but differentiates only in the angular
  variables.  Taking $\bB = [\bR , K]$ finishes the proof.
\end{proof}

We now proceed inductively to prove Theorem~\ref{thm:b-module-reg}; the case $\ell=0$ is handled above in the
proof of Theorem~\ref{thm:bpropagation}.  Setting $S=K$ or $S=R$ as
appropriate, we proceed using the commutants
\begin{equation*}
  W_{\ell} = S^{\ell} A^{*}A S^{\ell},
\end{equation*}
where $A$ is the commutant employed above.  Commuting $\dop$ with
$W_{\ell}$ yields
\begin{equation}
  \label{eq:commutator-with-factors}
  [\dop, W_{\ell}] = S^{\ell}[\dop, A^{*}A]S^{\ell} + [\dop,
  S^{\ell}]A^{*}A S^{\ell} + S^{\ell} A^{*}A [\dop, S^{\ell}].
\end{equation}
After applying the operator and pairing with $u$, the first term
yields the same terms in the argument above with $\ell=0$ (sandwiched
between factors of $S$).  Our aim is therefore to absorb or otherwise
bound the terms arising from commuting $\dop$ with $S^{\ell}$ and
pairing with $u.$

In the case of $S= R$, Lemma~\ref{lemma: b-module-commutators} allows
us to bound the remaining two terms by
\begin{equation*}
  \epsilon \norm{AR^{\ell}u}^{2} + C\epsilon^{-1}\sum_{j=0}^{\ell-1}
  \left( \norm{AR^{j}\dop u}^{2} + \norm{AR^{j}F_j u}^{2}\right),
\end{equation*}
$F_j \in \CI.$
The first term in this bound can be absorbed into the main term
arising from the commutator $[\dop, A^{*}A]$ in
equation~\eqref{eq:commutator-with-factors}, while the second term is
finite by the hypothesis on $\dop u$.  The third term is finite by the
inductive hypothesis.

We now consider the case of $S=K$.  By Lemma~\ref{lemma:
  b-module-commutators}, the remaining two terms are bounded by
\begin{equation*}
  \epsilon \norm{AK^{\ell}u}^{2} + C\epsilon^{-1}\sum_{j=0}^{\ell-1}
  \norm{A K^{j}\bB K^{\ell-1-j}u}^{2}.
\end{equation*}
Each of these terms will ultimately be absorbed into the main term by
choosing $\delta$ sufficiently small using the following lemma:

\begin{lemma}
  \label{lemma:bound-a-by-q}
  Suppose $A$ is defined as above and $Q \in \Psib^{s+1/2}$ is
  invariant with symbol
  \begin{equation*}
    \sigmab(Q) = \sqrt{2}|\tau|^{s+1/2}\delta^{-1/2}
    (\chi_{0}'\chi_{0})^{1/2} \chi_{1}\chi_{2}1_{\sgn\tau=\sgn\tau_{0}}.
  \end{equation*}
  There exists some $C$ (independent of $\delta$) and some
  $G \in \Psib^{s-1/2}$ so that
  \begin{equation*}
    \norm{Au}\leq C \left( \sqrt{\delta} \norm{Q u} +
      \norm{Gu} + \norm{u} \right).  
  \end{equation*}
\end{lemma}

\begin{proof}
  The proof is nearly identical to the one in
  Lemma~\ref{lemma:Bestimate}; because $\sigmab(A)$ is a multiple of
  $\sigmab(Q)$, we may write $A = CQ + R$, where $C \in \Psib^{0}$ has
  principal symbol
  \begin{equation*}
    \sigmab(C)= \left( 2 - \phi / \delta\right) \sqrt{\delta}/ \sqrt{2}.
  \end{equation*}
  Introducing the microlocalizer $G$ as in Lemma~\ref{lemma:Bestimate}
  finishes the proof.
\end{proof}

We now claim that we can bound $\norm{AK^{j}\bB K^{\ell-1-j} u}$ by
$\norm{AK^{\ell}u}$ and terms that are finite by the inductive
hypothesis.  Given this claim, Lemma~\ref{lemma:bound-a-by-q} then
allows us to absorb these terms into the main one by choosing $\delta$
sufficiently small, finishing the proof.

The rest of the section is devoted to the proof of the claim.  First
observe that because $\bB$ and $K$ are differential operators acting
only in the angular variables, we may replace them by scalar
operators in these variables, i.e., we may first bound
\begin{equation*}
  \norm{AK^{j}\bB K^{\ell-1-j}u} \leq C \sum_{|\alpha|\leq \ell} \norm{A\pd[\theta]^{\alpha}u},
\end{equation*}
where $C$ is independent of $u$.  All but the terms with
$|\alpha| = \ell$ are finite by the inductive hypothesis.  Because $A$
and $\pd[\theta]^{\alpha}$ are scalar operators, we again appeal to
the inductive hypothesis so that it suffices to bound
$\norm{\pd[\theta]^{\alpha}Au}$ for $|\alpha| = \ell$.  For $\ell =2m$
even, it suffices to control $\norm{\Lap_{\theta}^{m}Au} + \norm{Au}$,
while for $\ell = 2m+1$, the following lemma shows that it is enough
to control $\norm{K\Lap_{\theta}^{m}Au} + \norm{Au}$.

\begin{lemma}
  \label{lemma:back-to-K}
  There is a constant $C$ so that for any $u \in H_{\bo, g}^{1}$,
  \begin{equation*}
    \norm{\grad_{\theta}u}\leq C \left( \norm{Ku} + \norm{u}\right),
  \end{equation*}
  where the norms are taken with respect to $L^{2}$.
\end{lemma}

\begin{proof}
  Note that because $K$ contains only angular derivatives, $\norm{Ku}
  \leq C \norm{\grad_{\theta}u}$.  We then use that $\Lap_{\theta} =
  K^{2} - \beta K$ to see that
  \begin{align*}
    \norm{\grad_{\theta}u }^{2}&= \ang{\Lap_{\theta} u, u} =
                                 \ang{(K^{2}- \beta K)u, u} \\
                               &\leq \norm{K^{2}u} + \norm{Ku}\norm{u} \leq
                                 C \left( \norm{Ku}^{2} + \norm{u}^{2}\right).
  \end{align*}
\end{proof}

We now rely on Lemma~\ref{lemma:invariantcomm} and the
following observation: Because $K^{2} = \Lap_{\theta} + \beta K$,
\begin{equation}\label{K2m}
  K^{2m} = \Lap_{\theta} ^{m}+ \mathcal{L}_{2m}, \quad K^{2m+1} =
  \Lap_{\theta}^{m}K + \mathcal{L}_{2m+1},
\end{equation}
where $\mathcal{L}_{2m}$ is a constant linear combination of
$\Lap_{\theta}, \dots, \Lap_{\theta}^{m-1}$ and $\beta K, \Lap_{\theta}\beta
K, \dots, \Lap_{\theta}^{m-1}\beta K$, while $\mathcal{L}_{2m+1}$ is a
linear combination of $\beta \Lap_{\theta}, \dots, \beta
\Lap_{\theta}^{m}$ and $K, \Lap_{\theta}K, \dots, \Lap_{\theta}^{m-1}K$.

For $\ell = 2m$, we obtain (using Lemma~\ref{lemma:invariantcomm})
\begin{equation*}
  \norm{\Lap_{\theta}^{m}Au} = \norm{A\Lap_{\theta}^{m}u} \leq
  \norm{AK^{2m}u} + \norm{A\mathcal{L}_{2m}u},
\end{equation*}
where the second term is finite by the inductive hypothesis.  Likewise
for odd
$\ell = 2m+1$, 
\begin{equation}
  \norm{\Lap_{\theta}^{m}K Au} = \norm{A K  \Lap_{\theta}^{m}u} \leq
   \norm{AK^{2m+1}u} + \norm{A \mathcal{L}_{2m+1}u},
\end{equation}
where again by the inductive hypothesis the last term is
finite.  This finishes the proof of the claim and thus the proof of
Theorem~\ref{thm:b-module-reg}.

\section{Geometric improvement}
\label{sec:geom}

In this section we prove the second part of
Theorem~\ref{theorem:structure}, i.e., we show that the part of the
singularity of the fundamental solution lying on the diffracted wave
front $\diff$ and away from the geometrically  propagated light cone
$\geom$ is $1-0$ derivatives
smoother than the singularity along $\geom.$

There are two main steps to this argument.  In the first
(Section~\ref{sec:edgeprop}), we describe the propagation of edge
regularity, which allows us to propagate coisotropic regularity along
the geometric geodesics under a ``non-focusing'' condition.  In the
second part (Section~\ref{sec:global-reg-of-fundsoln}), we show that
we can apply the arguments of the first to a propagator.

\subsection{Propagation of edge regularity}
\label{sec:edgeprop}

In this section we establish the propagation of edge regularity.  The
propagation argument in this setting is somewhat less sensitive to
lower-order terms and so we are able to work with the second order
operator in this section.

Let $P$ be an operator satisfying the Klein-Gordon Hypotheses in
Section~\ref{section:KG}; recall that this means
\begin{equation}
  P\psi= -(\pa_0+i\frac{\charge}{r})^2 + \sum \pa_j^2 -m^2
  -i\frac{\charge}{r^2} \begin{pmatrix} 0 &
    \sigma_r\\ \sigma_r & 0 \end{pmatrix}
+\bR
\end{equation}
with
\begin{equation}\label{Rform2}
\bR=\charge \frac{\mathbf{W}_0}{r}+ \mathbf{W}_1^\alpha \pa_\alpha+\mathbf{W}_2,
\end{equation}
where the $\mathbf{W}_\bullet^\bullet$ coefficients are smooth but non-scalar.

As before, let $X=[\RR^3; 0]$ and $M = [\RR^{1+3}; \RR_{t}\times
\{0\}]$.  We now view $P$ as an operator in the \emph{edge} calculus on $M:$
$$
P \in r^{-2} \Diffe^2(\RR\times X)
$$
with
$$
\sigmae(P) = \frac{\lambda^2-\xi^2-\abs{\zeta}^2_{S^2}}{r^2}.
$$
The associated Hamilton vector field is then
$$
\hamvf= \frac{2}{r^{2}}\left(\left( \xi^{2} +
    \abs{\zeta}_{S^{2}}^{2}\right)\pd[\xi] + \xi \lambda \pd[\lambda]
  + \xi r\pd[r] - \lambda r\pd[t] \right) - \frac{1}{r^{2}}\hamvf_{S^{2}},
$$
where $\hamvf_{S^2}$ denotes the geodesic flow in $(\theta,\zeta) \in
T^*S^2.$  Let $\Sigma\subset \Sestar (\RR\times X)$ denote the
characteristic set of $P$.  

Recall that we have defined the edge Sobolev spaces be defined with respect to the
  b-weight as in \cite{Melrose-Wunsch1}, \cite{MVW1}.  Thus
$$
L^2_g  = r^{-3/2} \He^0.
$$
Likewise this is the scale on which we measure Sobolev-based edge
wavefront set $\WFe^*.$

We let $\module$ denote the graded module generated by angular
derivatives $\nabla_{S^2}.$ Let $\algebra^*$ denote the filtered
algebra over $\Psie^0(M)$ generated by $\module.$ Hence $\algebra$ is locally
generated by the operators $D_{\theta_j}.$

We will additionally be interested in testing for conormal regularity
along $N^*(\{t=r+r'\})$.  In addition to iterated regularity under vector
fields $D_{\theta}\in \module,$ this involves regularity under the operator
\begin{equation}\label{R}
  R=(t-r')D_t+rD_r;
\end{equation}
cf.\ \eqref{Rdef} where this operator appears with $r'=0.$  We will
often take advantage of time-translation invariance and tacitly set $r'=0$
in computing with this vector field.

The fact that, unlike the $D_{\theta_j}$'s, $R$ is not an edge vector field entails some
minor technical complication in what follows.

The commutator properties of $P$ with
$R$ and with the generators of $\algebra$ play an important role in
what follows.  As we are working in a
simpler geometric setting than that of \cite{MVW1}, we revert to the
simple expedient of using $\Lap_\theta,$ the angular Laplacian, as a
test operator for regularity in $\algebra.$
\begin{lemma}
  \label{lemma:edge-commutation}
  $\displaystyle [P,\Lap_\theta]=Q_0 D_r +Q_{1}D_{t} +r^{-2}Q_2$ and
  $\displaystyle [P,R]=-2i P +\frac{1}{r} Q_{3}$
  where
  \begin{equation*}
    Q_{j} \in \Diff^{1}(S^{2}), \quad j= 0, 1, 2,\quad
    \text{and}\quad Q_{3} \in \Diffe^{1}(M).
  \end{equation*}
\end{lemma}
\begin{proof}
All terms in the model operator $P-\bR$ with exact Coulomb potential
(see \eqref{Pform})
commute with $\Lap_\theta,$ except for the matrix valued term
$$
-i\frac{\charge}{r^2} \begin{pmatrix} 0 &
    \sigma_r\\ \sigma_r & 0 \end{pmatrix};
  $$
  commuting this with $\Lap_\theta$ gives the $r^{-2} Q_2$ term above,
  while the terms in $\bR$ contribute to the remaining error terms in
  $[P,\Lap_\theta].$

Additionally, by exact scaling symmetry in $t,r$
  $$
[P+m^2-\bR, R]=(-2i)(P+m^2-\bR),
$$
hence lumping the remaining terms in the $r^{-1} Q_3$ term gives the
desired expression for $[P,R].$
  \end{proof}

As above, let the canonical one form on $\Testar M$ be
$$
\lambda \frac{dt}{r}+ \xi \frac{dr}{r}+ \zeta \cdot d\theta.
$$
Let
$$
\IC_\pm(t_0,\theta_0) \equiv \big\{ t=t_0,\ r=0,\ \theta=\theta_0,\
\lambda=\pm 1,\ \xi =\pm 1,\ \zeta=0\big\} \subset \Sestar_{\RR\times \pa
  X} (M),
$$
$$
\OG_\pm(t_0,\theta_0) \equiv \big\{ t=t_0,\ r=0,\ \theta=\theta_0,\
\lambda=\pm 1,\ \xi =\mp 1,\ \zeta=0\big\} \subset \Sestar_{\RR\times \pa
  X} (M).
$$
These are the endpoint of the closures of bicharacteristic reaching
the front face of the blowup (i.e., the origin in the blown-down
space) from the interior as time increases (``incoming'') resp.\ as
times decreases (``outgoing'').  Note indeed that $\IC \cup \OG$
accounts for all the radial points of the edge Hamilton vector field $\hamvf.$

We now let $\flowout$ denote the
backward resp.\ forward flowouts of boundary points: if $p = \IC_{\pm} (t_0, \theta_0),$ let
$$
\flowout_{I}(p) \equiv \big\{ t=t_0-r,\ r\in (0, \ep),\ \theta=\theta_0,\
\lambda=\pm 1,\ \xi =\pm 1,\ \zeta=0\big\} \subset \Sestar(M),
$$
and
if $p = \OG_{\pm} (t_0, \theta_0),$ let
$$
\flowout_{O}(p) \equiv \big\{ t=t_0+r,\ r\in (0, \ep),\ \theta=\theta_0,\
\lambda=\pm 1,\ \xi =\mp 1,\ \zeta=0\big\} \subset \Sestar(M).
$$
These are the unique interior bicharacteristics containing the
corresponding radial points in their closures.

We now state a theorem about propagation of edge wavefront set,
together with module regularity.
\begin{theorem}\label{theorem:edgepropagation}
Suppose $u\in H_e^{-\infty,l}(M)$ solves
$$
Pu=0
$$
with $P$ satisfying the Klein-Gordon Hypotheses from \S\ref{section:KG}.
\begin{enumerate}
\item
Let $m>l+1.$  Set $p=\IC_\pm(t_0,\theta_0).$  If $\flowout_I(p) \cap
\WFe^m(\Lap_{\theta}^{j}u)=\emptyset$ for all $j=0,\dots, k$ then
$$
p \notin \WFe^{m,l'}(\Lap_{\theta}^{j} u)
$$
for all $j=0,\dots, k$ and all $l'<l.$

\item
Let $m<l+1.$  Set $p=\OG_\pm(t_0,\theta_0).$  Let $U$ denote a
punctured neighborhood of $p$ in $\Sestar_{\RR\times \pa X} (\RR\times
X)$.  If $U \cap \WFe^{m,l} \Lap_{\theta}^{j}u=\emptyset$ for all
$j=0,\dots, k$ then
$$
p \notin \WFe^{m,l}(\Lap_{\theta}^{j} u)
$$
for all $j=0,\dots, k$.

\item\label{bdyprop} For all $m,$ and $l'\leq l,$ $$\WFe^{m,l'}u \cap \Sestar_{\RR
    \times \pa X} (M)$$ is a union of maximally extended
  null bicharacteristics.

\item Suppose additionally that $R^j u \in H_e^{-\infty,l}$ and that
  $p \in \OG_{\pm}(t_0,\theta_0)$ has a neighborhood $U \subset
  \Sestar_{\RR \times \pa X}(\RR \times X)$ such that $U \cap
  \WFe^{m,l} u \subset \OG.$  Then for $0 \leq j'\leq j,$ $p \notin
  \WFe^{M,l}(R^j u)$ for $j \in \NN,$ provided $M \leq m-j$ and $M
  \leq l+1.$
\end{enumerate}
\end{theorem}
As shown \cite[Section 6]{Melrose-Wunsch1}, the propagation along null
  bicharacteristics within $\pa M$ (part (\ref{bdyprop}) above)
connects points in $\IC$ and points in $\OG$ lying over points
$\theta_0,\theta_1$ that are separated by geodesics of length $\pi$
with respect to the metric on $\pa X.$  Here $\pa X$ is simply $S^2,$
so this means that the propagation is from a point $\theta_0$ to its
antipodal point $\theta_1=-\theta_0.$

    The theorem thus says that regularity propagates
\begin{enumerate}
\item From the interior of $M$ into incoming radial points in $\pa M$
  (a.k.a.\ the lift of $r=0$ under blowup) along bicharacteristics,
  above some threshold regularity dictated by the weight in $r$
  \item Across $\pa M$ along bicharacteristics, from incoming radial
    points to antipodal outgoing radial points (instantaneously)
    \item From outgoing radial points back into the interior of $M$,
      up to some threshold regularity dictated by the weight in $r.$
    \end{enumerate}
    Owing to the limits in regularity in the outgoing part of the
    theorem (which is typical in radial point problems---cf.\ \cite{MR95k:58168}), this result does not in fact say that regularity arrives
    at the boundary, propagates across it, and leaves, at any given
    Sobolev order.  Obtaining regularity (and, ultimately,
    conormality) at the outgoing wavefront
    will require subtler arguments involving $D_\theta$ and $R$
    regularity, hence the need for these factors to propagate through
    our estimates as well.
\begin{proof}
  The proof is the same as that of Theorem~8.1 of
  \cite{Melrose-Wunsch1}.  We sketch the first part here in order to
  verify that the passage to a slightly different class of operators
  under consideration here (with cross term involving $r^{-1} \pa_t,$
  inverse square potential terms, a principally scalar system with a
  large anti-self-adjoint $0$'th order term) do not vitiate the
  arguments used there.

Let $p=\IC_\pm(t_0, \theta_0).$  We will begin by sketching the proof of the
following propagation result, which gives the first part of the
theorem up to the factors in of $\Lap_\theta^k$
\begin{propest}\label{propest:1}
If $m'>l'+1/2,$ $u \in \He^{-\infty,l'}(M),$ $p \notin
\WFe^{m',l'} (u),$ and $\flowout_I(p)\cap \WF^{m'+1/2}
u=\emptyset,$ then $p \notin \WFe^{m'+1/2,l'}u.$
\end{propest}
To establish Propagation Estimate~\ref{propest:1} we choose $A\in
\Psie^{m',l'+1/2}$ such that
\begin{equation}\label{maincommutator1}
P^* A^*A-A^*A P = \pm (A')^* (A')\pm \sum B_j^* B_j + E+K +F,
\end{equation}
where
\begin{enumerate}
\item $A,$ $A'$ are microsupported near $p.$
\item 
$A' \in \Psie^{m'+1/2, l'+3/2}$ with 
$\sigmae(A') = \sigmae(A) \cdot (\pm (m'+l'+1/2)\xi)^{1/2}$
\item $E \in \Psie^{2m'+1,2l'+3}$ and $\WF' E$ is in an arbitrarily small neighborhood of a single
  point in $\flowout_I(p).$
\item $K   \in \Psie^{2m'+1,2l'+3}$ and $\WF' K \cap \Sigma=\emptyset$
\item $F$ is of lower order, lying in $\Psie^{2m',2l'+3}.$  (Note
  that it is only the pseudodifferential order that is lower, not the weight.)
\end{enumerate}
Notwithstanding that our convention for b-Sobolev spaces is to base
them on $L^2_b,$ the adjoints above are all taken with respect to the
inner product on $L^2_g,$ as we will use this inner product (with
respect to which $P$ is mostly self-adjoint) in making a pairing
argument below.

The operator $A$ is constructed roughly as follows: if $m+l > 0$, then
$\hamvf (\lambda^{m}r^{l}) = (m+l)\xi \lambda^{m}r^{l}$, so that if
$\chi(s) \equiv 0$ for $s<0$ and $\chi(s) \equiv 1$ for $s \geq 1$,
$\hamvf \left(\chi (\pm \lambda) \chi (\pm \xi) (\pm
  \lambda)^{m}r^{l}\right)$ has the same sign as $\xi$ (the $\pm$ used
here).  We can localize in the $\theta$ variable as in the
more general treatment in \cite[Section 6]{Melrose-Wunsch1} by using a
function given, in our blown-down Euclidean coordinates
$(x,\underline{\xi}) \in T^*\RR^3$ by cutting off
$-\hat{\underline{\xi}}$ to lie in a small neighborhood of any desired
$\theta_0$; such a cutoff manifestly commutes with the Hamilton flow,
and is shown in \cite{Melrose-Wunsch1} to lift to be a smooth symbol
on $\Tbstar M.$ If a geodesic arrives at the origin, then since it is
oriented radially, its angle of
arrival $\theta \in S^2$ is manifestly $-\hat{\underline{\xi}},$ hence
we have achieved an angular localization.
Finally, a cutoff in
$\smallabs{\zeta}/\smallabs{\lambda}$ has the same sign as the signed
terms listed above.  Thus, the product of these cutoffs localizing in
$\theta,\zeta$ with
$\chi (\pm \lambda) \chi (\pm \xi) (\pm \lambda)^{m}r^{l}$ may be
quantized to give an $A$ with the desired properties -- see Lemma~7.1
of~\cite{Melrose-Wunsch1} for details. 

We remark that the system
under consideration here may be treated as a scalar equation from the
point of view of the positive commutator argument because the
principal symbol of $P$ in the edge calculus is scalar.  In
particular, the anti-self-adjoint term in $P$,
$$
-i \frac{\charge}{r^2}
\begin{pmatrix} 0 &
  \sigma_r\\ \sigma_r & 0 \end{pmatrix},
$$
which is large enough to disrupt commutator arguments in the b-calculus,
lies in $\Psie^{0, 2},$ hence in the twisted commutator
$P^*A^*A-A^*A P$ gives rise to a term in $\Psi^{2m', 2l'+3}$ which
may be included in the lower-order error term $F$ above.

Now the propagation argument follows by pairing the equation
\eqref{maincommutator1} with $u,$ \emph{using the metric inner
  product} $r^{2} \, dr \,d\theta.$  The left-hand-side is zero,
by integration by parts.  Technically, in fact, we
require an approximation of $A$ by operators in $\Psie^{-\infty,
  l'+1/2}$ in order to justify this integration by parts---see
\cite{Melrose-Wunsch1} for details of this approximation process,
which involve a family of smoothing operators $A_\delta$ with a
further parameter approximating $A$ as $\delta\downarrow 0.$

The terms of the right hand side of the pairing are then as follows.
The term $\norm{A'u}_{L^2_g}^2$ is precisely what we need to control:
note that in terms of the b/edge-volume form $dr/r \, dt \,
d\theta=r^{-3} dV_g,$ this term is of the
form
$$
\norm{r^{3/2} A'}^2_{L^2_e},
$$
hence controls $\WFe^{m'+1/2, l'}u.$
The terms $\norm{Bu}^2$ have the same sign, and hence may be
dropped.  The term with $E$ is controlled by our incoming wavefront
set hypothesis.  The term with $K$ is controlled by microlocal
elliptic regularity. And the term with $F$ is controlled by our
assumption $p \notin\WFe^{m',l'} u.$  This concludes our proof
of Propagation Estimate~\ref{propest:1}.

Now we can employ Propagation Estimate~\ref{propest:1} iteratively to
obtain the first part of the theorem, in the case $k=0.$
We know a priori that $u \in \He^{q,l}$ for \emph{some} $q;$ if
$q>l+1/2$ we may immediately iterate the propagation estimate to
obtain the result of the theorem.  If not, we must artificially lower
our $l$ to some $l'<q-1/2$ in order to start the iteration.  In this
case, however, an interpolation argument still recovers the result but
ends up with $l'=l-\ep$ for any desired $\ep>0$---see Figure~1 of
\cite{Melrose-Wunsch1} and related discussion.

To include the module regularity in the first part of the theorem, we
proceed inductively, employing the same commutant as above and considering
the twisted product
\begin{equation}\label{Atheta}
\begin{aligned}
  &P^{*} \Lap_{\theta}^{k} A^{*}A \Lap_{\theta} ^{k}-
    \Lap_{\theta}^{k}A^{*}A\Lap_{\theta}^{k}P \\
  &\quad = [P^{*}, \Lap_{\theta}^{k}] A^{*}A \Lap_{\theta} ^{k}-
    \Lap_{\theta}^{k}A^{*}A[\Lap_{\theta}^{k}, P] \\
  &\quad\quad + \Lap_{\theta}^{k} \left( P^{*}A^{*}A - A^{*}AP\right) \Lap_{\theta}^{k}.
\end{aligned}
\end{equation}
The last term gives rise to similar terms as in the propagation
estimate (with $u$ replaced by $\Lap_{\theta}^{k}u$) and so
allow us to control $\WFe^{m'+1/2,l'}\Lap_{\theta}^k u$.
Indeed, together with terms that are finite by induction, this term controls
$\sum_{\smallabs\alpha \leq 2k} \norm{D^{\alpha}_\theta A' u}^2$ with
$A'$ as before.
The first
two terms on the RHS of \eqref{Atheta} can then be absorbed into this main
term (modulo inductively finite terms);
here we use the fact that  while having the same order, these error terms have a smaller
$r$ weight.

are controlled by the induction hypothesis (together with
the description of $[P, \Lap_{\theta}]$ given by
Lemma~\ref{lemma:edge-commutation}).

The remaining parts of the theorem follow in an essentially identical
way to those of Theorem~8.1 of~\cite{Melrose-Wunsch1}, and similar to
the arguments given above.
\end{proof}

\subsection{Global propagation of coisotropic regularity}
\label{sec:global-reg-of-fundsoln}

Our aim in this section is to apply
Theorem~\ref{theorem:edgepropagation} to the solution of
$(i\dirac_{\pot}-m)u = 0$ with initial condition $\psi_{0}\delta_{y}$ and
verify that the diffracted wavefront is $1-0$ orders smoother than the
propagated one.

The sketch of the proof is as follows:
For each time, the solution is a distribution $u$ of Sobolev order
$-3/2 - 0$.  An angularly smoothed version of the solution,
$\ang{\Lap_\theta}^{-M}u $ (for $M\gg 0$) is, by contrast, a
distribution of order $-1/2 - 0.$ (In the language below, $u$
has global nonfocusing regularity of order $-1/2-0$).  Additionally, at a
point on the diffracted front away from the propagated light cone,
Theorem~\ref{theorem:edgepropagation} shows that $u$ has infinite
order coisotropic regularity with respect to a \emph{weaker} Sobolev norm,
i.e., $D_\theta^\alpha u \in H^{k}$ for all $\alpha,$ with $k$ fixed.
Interpolation of the coisotropic regularity with the angular smoothing
effect then shows that in fact $u$ has infinite order
coisotropic regularity with respect to the \emph{better space} (up to
an $\epsilon$ loss) and therefore
is locally a distribution of Sobolev order $-1/2-0$ enjoying coisotropic
regularity.  Additionally propagating powers of $R=tD_t+rD_r$ through
the evolution then suffices to show that $u$ enjoys Lagrangian
regularity with respect to $H^{-1/2-0}$ along the diffracted wave, as desired.



\begin{definition}
Fix a Hilbert space $\hilbert$ and a set $K\subset \Sbstar (M).$

A distribution on $\RR \times X$ enjoys \emph{coisotropic regularity} (of order $2N$) with respect to $\hilbert$ on $K$
if there exists a properly supported operator $A \in \Psib^0(M),$ elliptic on $K,$ such that 
  $$
(\Id+\Lap_\theta)^NAu  \in \hilbert.
$$  

A distribution on $M$ is \emph{nonfocusing} with respect to
$\hilbert$ on $K$ if there exists a properly supported operator $A \in \Psib^0(M^\circ),$ elliptic on $K$ and there exists $N \in \NN$ such that 
  $$
Au = (\Id+\Lap_\theta)^N u',\quad u' \in \hilbert.
$$

We also make analogous definitions at the level of Cauchy data, i.e.,
distributions on $X$: if $\hilbert'$ is a Hilbert space of
distributions on $X,$ and $K \subset \Sbstar X$,
a distribution on $X$ enjoys coisotropic regularity (of
order $2N$) with respect to $\hilbert$ on $K$
if there exists a properly supported operator $A \in \Psib^0(X),$ elliptic on $K,$ such that 
  $$
  (\Id+\Lap_\theta)^NAu \in \hilbert.
$$  

A distribution on $X$ is nonfocusing with respect to $\hilbert'$ on $K$ if there exists a properly supported operator $A \in \Psib^0(M),$ elliptic on $K$ and there exists $N \in \NN$ such that 
  $$
Au = (\Id+\Lap_\theta)^N u',\quad u' \in \hilbert'.
$$
\end{definition}
One could of course refine the nonfocusing definition by
specifying in the terminology the power $N$ for which it holds, but in practice we will
be concerned with the union of this nonfocusing condition over all
possible $N.$  In this paper, moreover, we will mainly be concerned with
localizing over a particular set in the $t$ variable, but will neither
localize in other variables nor microlocalize, hence the subtleties of
microlocalizing in the b-calculus are moot.

In practice, it is convenient to take $\hilbert$ to be
$L^2_{\loc}(\RR; \dom^s)$ (where we will drop the ``loc'' from now on
as global estimates in time play no role here).  This formulation is convenient for duality
arguments owing to the sensible behavior of these spaces near the
origin, but away from the origin, we remark that nonfocusing with
respect to $\dom^s$ is in fact equivalent to nonfocusing with
respect to $H^s.$

Note also that we may equivalently test for coisotropic regularity with powers of Dirac's angular operator $K$ instead of powers of
$\Lap_\theta:$ Since
$$
K^2-\beta K=\Lap_\theta,
$$
regularity under powers of $K$ up to $2N$ yields regularity under
$(\Id+\Lap_\theta)^N;$ conversely, regularity under $(\Id + \Lap_\theta)^N$
yields $K$-regularity by ellipticity of $\Lap_\theta$ in the angular
variables.  Likewise the condition of nonfocusing can be recast as
lying in the range of sums of powers of $K,$ and we will use this
alternative version below.

\begin{lemma}\label{lemma:coiso1}
Let
  $$
(i\dirac_\pot-m)u=0.
$$
If for some $\ep>0,$ $u$ enjoys coisotropic regularity of order $N$ with respect to $L^2_{\loc}(\RR; \dom^s)$ on
$(-\ep,\ep)_t \times X$
then $u$ enjoys coisotropic regularity of order $N$ with respect to
$L^2_{\loc}(\RR; \dom^s)$ globally on $M.$

If for some $\ep>0,$ $u$ enjoys the nonfocusing condition with respect to $L^2_{\loc}(\RR; \dom^s)$ on
$(-\ep,\ep)_t \times X$
then $u$ enjoys the nonfocusing condition with respect to
$L^2_{\loc}(\RR; \dom^s)$ globally on $M.$

The conditions of coisotropic regularity resp.\ nonfocusing w.r.t.\
$L^2_{\loc}(\RR; \dom^s)$ are moreover equivalent to the conditions of
coisotropic regularity resp.\ nonfocusing of the Cauchy data $u(t_0)$
(for any $t_0$) w.r.t\ $\dom^s.$
\end{lemma}
\begin{proof}
We begin with coisotropic regularity.
 By
Lemma~\ref{lemma:diracenergy},
\begin{equation}\label{Kenergy}
\frac d{dt} \norm{K^j u}_{\dom^s}^2 \lesssim
\abs{\smallang{[K^j, \ellipticdirac]u,K^j u}_{\dom^s}}.
\end{equation}
By Lemma~\ref{lemma: b-module-commutators}
$[\ellipticdirac, K^{j}]$ is a linear combination of terms of the
    form $K^{j'}\bB K^{j-1-j'}$, where $j' = 0, 1, \dots, j-1$ and
    $\bB \in \Diffb^{1}$ only differentiates in the angular variables.
    Thus by Lemma~\ref{lemma:back-to-K} and the following
    discussion,
    we may bound
    $$
\abs{\ang{K^{j'}\bB K^{j-1-j'}u, K^j u}_{\dom^s}} \lesssim \sum_{j'
  \leq j} \norm{K^{j'} u}_{\dom^s}^2
    $$

    Thus by Cauchy--Schwarz and  Gronwall, \eqref{Kenergy} yields inductively for all $T$,
$j,$
$$
\sum_{j'=0}^j \norm{K^{j'} u(t)}_{\dom^s} \leq C_{T,j} \sum_{j'=0}^j
\norm{K^{j'} u(0)}_{\dom^s},\quad \smallabs{t}\leq T.
$$
This shows that coisotropic regularity of the Cauchy data propagates,
and moreover that coisotropic regularity of the Cauchy data implies
$L^\infty_{\loc}(\RR; \dom^s)$ coisotropic regularity of the spacetime
solution.  Conversely, knowing merely $L^2 \dom^s$ coisotropic
regularity of the spacetimes solution implies that for a.e.\ $t,$ the
Cauchy data $u(t)$ enjoys coisotropic regularity, which then in turn
propagates to yield $L^\infty_{\loc}(\RR; \dom^s)$ spacetime
regularity.  This finishes the proof of the lemma for coisotropic regularity.

We now turn to nonfocusing.  We note that by the coisotropic results,
applied backwards in time, if we let $\hilbert$ denote the Hilbert
space with squared norm
\begin{equation}\label{H}
\sum_{j'=0}^j  \norm{K^{j'}u}_{\dom^s}^2,
\end{equation}
then we have estimated
\begin{equation}\label{propagatoronhilbert}
U(t) \hilbert \to L^\infty_{\loc}(\RR; \hilbert).
  \end{equation}
In particular, for fixed $t,$ $U(-t)$ is bounded $\hilbert \to \hilbert.$
Thus (by unitarity on $\dom^s,$) $U(-t)=U(t)^*: \hilbert^* \to
\hilbert^*,$ with dual spaces taken with respect to $\dom^s$ inner
product.  By the Riesz lemma,
\begin{equation}\label{Hstar}
\hilbert^*= \sum_{j'=0}^j K^{j'} \dom^s.
\end{equation}
This is just the space of Cauchy data nonfocusing with respect to
$\dom^s,$ hence the nonfocusing of Cauchy data is preserved under propagation.
Moreover, uniformity in $t$ of the maps $\hilbert\to
\hilbert$ show the uniformity in $t$ of the dual maps, hence yield the
equivalence with nonfocusing with respect to $L^2_{\loc}(\RR; \dom^s)$
as above in the coisotropic regularity case.
      \end{proof}

In order to show conormal regularity of the diffracted wavefront, it
is useful to have a refinement of Lemma~\ref{lemma:coiso1} that
additionally allows powers of the scaling operator $R.$
\begin{lemma}\label{lemma:coiso2}
Let
  $$
(i\dirac_\pot-m)u=0.
$$
Fix $k \in \NN.$
If for some $\ep>0,$ $u, Ru, \dots, R^ku$ enjoy coisotropic regularity of order $N$ with respect to $L^2_{\loc}(\RR; \dom^s)$ on
$(-\ep,\ep)_t \times X$
then $u, Ru, \dots, R^ku$ enjoy coisotropic regularity of order $N$ with respect to
$L^2_{\loc}(\RR; \dom^s)$ globally on $M.$

If for some $\ep>0,$ $u,\dots, R^ku$ enjoy the nonfocusing condition with respect to $L^2_{\loc}(\RR; \dom^s)$ on
$(-\ep,\ep)_t \times X$
then $u, \dots, R^ku $ enjoy the nonfocusing condition with respect to
$L^2_{\loc}(\RR; \dom^s)$ globally on $M.$

The conditions of coisotropic regularity resp.\ nonfocusing w.r.t.\
$L^2_{\loc}(\RR; \dom^s)$ for $R^j u$ are moreover equivalent to the conditions of
 coisotropic regularity resp.\ nonfocusing of the Cauchy data $\tR^j
 u(t_0),$ where $\tR = -t\ellipticdirac +r D_r$
(for any $t \in \RR$) w.r.t\ $\dom^s.$
\end{lemma}
\begin{proof}

  To obtain the propagation of coisotropic regularity of order $N$, we recall from
  Lemma~\ref{lemma: b-module-commutators} that 
\begin{equation}\label{Rku}
\dop R^k u = \sum_{k'=0}^{k-1} \CI R^{k'} u
\end{equation}
(with the $\CI$ terms non-scalar).
Thus, if $\hilbert$ is defined as in \eqref{H}, and if we inductively
assume that $u,\dots, R^{j-1} u$ enjoy coisotropic regularity, i.e.,
lie in $L^\infty_{\loc}(\RR; \hilbert),$ then
$$
\dop R^j u=\sum_{j'=0}^{j-1} \CI R^{j'} u \in L^\infty(\RR;\hilbert).
$$
Moreover if we assume that $R^j u$ has coisotropic regularity
initially, then it has initial data in $\hilbert.$ Duhamel's theorem
(employed with values in $\hilbert$) and \eqref{propagatoronhilbert}
  then imply that
$$
R^j u \in L^\infty(\RR; \hilbert) 
$$
as well; this inductively shows propagation of coisotropic regularity for $R^j u.$

The equivalence with the Cauchy data statement simply follows from the fact
that 
$$
Ru =\tR u
$$
for solutions of the Dirac equation.

To obtain the propagation of nonfocusing for $R^j u,$ where we have to
dualize in powers of $K$ but not in powers of $R,$ we apply the same
argument as above but with solutions in $L^\infty(\RR; \hilbert^*)$
rather than $L^\infty(\RR; \hilbert):$ we inductively show that $R^j u
\in L^\infty_{\loc}(\RR; \hilbert^*)$ for each $j \in \NN.$

\end{proof}
    
\begin{lemma}\label{lemma:fundsoln}
  Fix a 4--spinor $\psi_0$ and a point $x_0 \in \RR^3.$ Then the solution $u$ to the Dirac equation with initial data
  $$
\delta(x-x_0) \psi_0
$$
is in $\mathcal{C}(\RR; \dom^{-3/2-0}),$
and enjoys nonfocusing (on all of $M^\circ$) with respect to $\dom^{-1/2-0}.$
\end{lemma}
\begin{proof}
  This is essentially a vector-valued version of \cite[Lemma 16.1,
  Proposition 16.2]{Melrose-Wunsch1}. We first note that by energy
  conservation (see \S\ref{section:energy}), $u \in \mathcal{C}(\RR; \dom^{-3/2-0})$ since $\delta \in H^{-3/2-0}.$
  On the other hand, given any $k,$ for $N$ large,
  $$
(\Id+\Lap_\theta)^{-N} \delta(r-r_0) \delta(\theta-\theta_0) \in
\mathcal{C}^k(S^2; H^{-1/2-0}(\RR_+)),
$$
hence taking $k\gg 0$ yields
$$
(\Id+\Lap_\theta)^{-N} u(0) \in \dom^{-1/2-0}.
$$
This suffices to establish nonfocusing at $t=0$ and hence globally in
time, by Lemma~\ref{lemma:coiso1}.
\end{proof}

Finally, we consider the regularity of the solution on the strictly
diffracted wavefront $\diff
\backslash \geom.$  Let $u$
denote the solution with initial data   $\delta(x-x_0) \psi_0,$ where
$x_0=(r_0,\theta_0)$ in polar coordinates.  For $t_0>r_0,$ consider any
point $(r=t_0-r_0, \theta)$ with $\theta\neq -\theta_0$ and let $U$ be a
neighborhood of this point in $X^\circ$ disjoint from
$\pi(\geom)=\{\smallabs{x-x_0}=t\}$ for $t \in I\equiv (t_0-\ep, t_0+\ep).$
  By Lemma~\ref{lemma:fundsoln}, $u$ is
nonfocusing in $T^* (I\times U)$ (or, indeed, globally) relative to $L^2\dom^{-1/2-0}.$  On the other
hand, we now apply the edge propagation theorem
(Theorem~\ref{theorem:edgepropagation})  to the solution
$\Theta_{-3/2+\ep} u,$ which lies in $\mathcal{C}^0(\RR; \dom^0),$
hence in particular, say, in $L^2_\loc(M).$  Thus the
edge regularity hypotheses of the edge propagation theorem are
satisfied (with $l=0$), and we conclude, also using
Proposition~\ref{proposition:interiorcoiso} for propagation into $r>0,$ that for some fixed $M,$ for
all $k \in \NN,$ $\WF^M (\Lap_\theta^k \Theta_{-3/2-\ep} u)\cap
T^*(M^\circ)$ is disjoint from the strictly diffractive flowout from the
origin $$N^*\{r=t-r_0\}\cap \{ \theta\neq -\theta_0\}.$$
In particular, then,
since no points in $T^*(I\times U)$ are geometrically related to the initial
singularity, $u$ (which differs from
$\Theta_{3/2+\ep}(\Theta_{-3/2-\ep} u)$ by a smooth error) enjoys
\emph{coisotropic regularity of every order} relative to some Sobolev space
$H^{M'}$ on $I\times U.$  By an
interpolation argument \cite[Section~13]{MVW1}, a distribution that is
nonfocusing relative to $H^s$ and enjoys infinite order coisotropic
regularity relative to some fixed $H^k$ in fact lies in $H^{s-0},$
hence $u$ enjoys this regularity over $I\times U$ (and it moreover also
enjoys iterated regularity under $K$ relative to these spaces).  This
proves that the fundamental solution $u$ lies in
$H^{-1/2-0}$ near $\diff \backslash \geom$ and moreover that
$\Lap_\theta^k u$ enjoys the same regularity for all $k.$

Finally, we show that the diffracted wave is a conormal singularity.  To
begin, we further analyze the singularity of the fundamental solution
for short time: since $Pu=0$ with $P=\Box$ modulo lower order terms,
we have energy estimates for $u$ for short time, and the parametrix
construction \cite[Theorem
29.1.1]{Hormander:vol4} applies, and shows that $u\in \mathcal{C}(\RR;
\dom^{-3/2-0})$ is conormal 
to $\smallabs{x-x_0}=\smallabs{t}$ whenever
$\smallabs{t}<\smallabs{x_0}.$ (Beyond this range of times, the support
reaches the singularity of the potential, which cannot be treated as a
perturbation any longer).
Consequently, as $N\to \infty,$ the angular smoothing of $u,$
$$
(\Id+\Lap_\theta)^{-N} u,
$$
approximates a sum of conormal distributions in $H^{-1/2-0}$ at the
hypersurfaces $r=r_0 \pm t.$ Since $R\equiv (t-r_0) D_t+ rD_r$ is
tangent to $\{r=r_0 - t\},$ for $t \in (0, r_0)$ the regularity of
this latter piece of the solution is unaffected by the iterated
application of $R.$ Thus for each $j \in \NN,$ $R^j u$ satisfies the
nonfocusing condition relative to $\dom^{-1/2-0}$ for
$t \in (0, r_0),$ microlocally away from the outgoing spherical wave
$N^*\{r=r_0+t\}.$ (See \cite[Lemma 16.1]{Melrose-Wunsch1} for details
of this computation.)  Note that we may microlocalize our solution
away from the outgoing spherical wave without changing the diffracted
wave (by the b propagation theorem), hence we may ignore this part of
the solution.

By Lemma~\ref{lemma:coiso2}, the nonfocusing condition persists for
all $t \in \RR.$ On the other hand,
Theorem~\ref{theorem:edgepropagation}
implies that along the
strictly diffracted wavefront (and for $r$ small), for every $j \in \NN,$ $R^j u$ enjoys
coisotropic regularity with respect to some fixed (but $j$-dependent)
Sobolev space $H^{-M(j)}.$ Once again, by interpolation, we then have
$R^j u \in H^{-1/2-0}$ along the strictly diffracted wavefront for every $j,$
and this, along with the coisotropic regularity and the equation
$Pu=0,$ establishes conormal regularity along the Lagrangian $\diff =
N^{*}\{r = t - r_{0}\}$ at points $\theta \neq -\theta_{0}$ (i.e.,
away from $\geom$).
  

\def\cprime{$'$} \def\cprime{$'$} \def\cprime{$'$}
  \def\polhk#1{\setbox0=\hbox{#1}{\ooalign{\hidewidth
  \lower1.5ex\hbox{`}\hidewidth\crcr\unhbox0}}} \def\cprime{$'$}
  \def\cprime{$'$} \def\cprime{$'$} \def\cprime{$'$} \def\cprime{$'$}
  \def\cprime{$'$} \def\cprime{$'$} \def\cprime{$'$} \def\cprime{$'$}
  \def\cprime{$'$} \def\cprime{$'$} \def\cprime{$'$} \def\cprime{$'$}
  \def\cprime{$'$} \def\cprime{$'$} \def\cprime{$'$} \def\cprime{$'$}
  \def\cprime{$'$} \def\cprime{$'$} \def\cprime{$'$} \def\cprime{$'$}
  \def\cprime{$'$} \def\cprime{$'$} \def\cprime{$'$} \def\cprime{$'$}
  \def\cprime{$'$} \def\cprime{$'$} \def\cprime{$'$}

\end{document}